\documentclass[11pt,reqno]{amsart}
\usepackage[utf8]{inputenc}
\usepackage[letterpaper,hmargin=1.0in,vmargin=1.0in]{geometry} 
\usepackage{amsmath,amsthm,amsfonts,amssymb,mathrsfs,stmaryrd,bigints}
\usepackage{mathtools}
\usepackage[usenames]{color}
\newtheorem{theorem}{Theorem}[section]
\newtheorem{thm}[theorem]{Theorem}
\newtheorem{lemma}[theorem]{Lemma}
\newtheorem{lem}[theorem]{Lemma}

\newtheorem{cor}[theorem]{Corollary}
\newtheorem{proposition}[theorem]{Proposition}
\newtheorem{prp}[theorem]{Proposition}
\newtheorem{observation}[theorem]{Observation}

\newtheorem{maintheorem}{Theorem}

\def\N{\mathbb{N}}

\def\P{\mathbb{P}}
\def\Z{\mathbb{Z}}
\def\R{\mathbb{R}}

\def\E{\mathbb{E}}

\newcommand{\cP}{\mathcal{P}}

\newcommand{\ce}{\mathcal{E}}
\newcommand{\cg}{\mathcal{G}}
\newcommand{\cU}{\mathcal{U}}
\newcommand{\cL}{\mathcal{L}}
\newcommand{\tr}{\mathrm{tr}}
\newcommand{\e}{\varepsilon}

\newcommand{\bn}{\mathbf{n}}

\usepackage[colorlinks=true,linkcolor=blue]{hyperref}

\begin{document}

\title[Large Deviation for Polymers]{Connecting Eigenvalue Rigidity with Polymer Geometry: Diffusive Transversal Fluctuations under Large Deviation}
\date{}

\author[R. Basu]{Riddhipratim Basu}
\address{R. Basu\\
  International Centre for Theoretical Sciences\\
  Tata Institute of Fundamental Research\\
  Bangalore, India
  }
  \email{rbasu@icts.res.in}

\author[S. Ganguly]{Shirshendu Ganguly}
\address{S. Ganguly\\
  Department of Statistics\\
 U.C. Berkeley \\
  Evans Hall \\
  Berkeley, CA, 94720-3840 \\
  U.S.A.}
  \email{sganguly@berkeley.edu}

\maketitle
\begin{abstract}
We consider the exactly solvable model of exponential directed last passage percolation on $\Z^2$ in the large deviation regime. Conditional on the upper tail large deviation event $\cU_{\delta}:=\{T_{n}\geq (4+\delta)n\}$ where $T_{n}$ denotes the last passage time from $(1,1)$ to $(n,n)$, we study the geometry of the polymer/geodesic $\Gamma_{n}$, i.e., the optimal path attaining $T_{n}$. We show that conditioning on $\cU_{\delta}$ changes the transversal fluctuation exponent from the characteristic $2/3$ of the KPZ universality class to $1/2$, i.e., conditionally, the smallest strip around the diagonal that contains $\Gamma_{n}$ has width $n^{1/2+o(1)}$ with high probability. This sharpens a result of Deuschel and Zeitouni (1999) \cite{DZ1} who proved a $o(n)$ bound on the transversal fluctuation in the context of Poissonian last passage percolation, and complements \cite{BGS17A}, where the transversal fluctuation was shown to be $\Theta(n)$ in the lower tail large deviation event. Our proof exploits the correspondence between last passage times in the exponential LPP model and the largest eigenvalue of the Laguerre Unitary Ensemble (LUE) together with the determinantal structure of the spectrum of the latter. A key ingredient in our proof is a sharp refinement of the large deviation result for the largest eigenvalue \cite{Jo99, sep98}, using rigidity properties of the spectrum, which could be of independent interest.\end{abstract}
\tableofcontents
\section{Introduction}
Last passage percolation (LPP) models on $\Z^2$ are paradigm examples of models believed to exhibit the features of the Kardar-Parisi-Zhang universality class. In such models, one studies the weight and geometry of the maximum weight directed path (called henceforth a polymer or a geodesic) between two far away points in a field of i.i.d.\ weights on vertices of $\Z^2$. Starting with the breakthrough work of Baik, Deift and Johansson \cite{BDJ99}, last two decades have seen  an explosion of results in studying the so-called exactly solvable models in this class; where using some remarkable bijections leading to exact formulae have been used to obtain the typical longitudinal and transversal fluctuation exponents of $1/3$ and $2/3$, and also to establish universal scaling limits that appear in random matrix theory. 
Historically, progress was made on understanding the large deviation of the maximal weight  (henceforth called last passage time) around the late nineties \cite{VerKer77, LogShep77, DZ2, sepLDP} before the typical fluctuation was rigorously understood. However, finer results about how the geometry of the polymer changes under the large deviation regime was not available until very recently. {With }Allan Sly the authors  recently showed in \cite{BGS17A} that for a large class of models, conditioning on the lower tail large deviation, i.e., on the event that the last passage time is macroscopically smaller than its typical value (see precise definitions below) the transversal fluctuation exponent changes to $1$; that this the polymer from $\mathbf{1}$ to $\mathbf{n}$ (for $r\in \Z$, the point $(r,r)$ will be denoted by $\mathbf{r}$) is unlikely to be contained in any strip of width $o(n)$. In this paper we continue the program of understanding the geometry of polymers under large deviation and focus on upper tail large deviation regime; i.e., the event where the last passage time is macroscopically larger than typical. One might heuristically expect that conditioning on the upper tail large deviation event will reduce the transversal fluctuation exponent (from the typical value $2/3$). For the exactly solvable model of LPP with i.i.d.\ exponential passage times, we show that it is indeed the case, and obtain the exact value of the transversal fluctuation exponent to be $1/2$. 

Unlike the approach in \cite{BGS17A}, our approach here crucially uses the exactly solvable nature of exponential LPP. Indeed, our proof is based on the fact that the last passage time between two vertices in exponential LPP has the same distribution as the largest eigenvalue of a certain Wishart matrix, whose eigenvalues form a determinantal point process (LUE). Using recent breakthroughs in understanding of the rigidity of eigenvalues we obtain sharp asymptotics of large deviation probabilities for the largest eigenvalue of LUE, improving earlier results. This in turn, together with the above correspondence leads to sharp estimates on the transversal fluctuation of the polymer under upper tail large deviation. We now move towards precise definition and statement of main results.

\subsection{Definitions and Main Results}\label{setting}

We consider the following last passage percolation (LPP) model on $\Z^2$: let $\{X_{v}:v\in \Z^2\}$ denote a field of i.i.d.\ $\mbox{Exp}(1)$ random variables. For any up/right path $\gamma$, the weight of the path, denoted $\ell(\gamma)$ is given by the sum of the weights on $\gamma$, i.e., 
$$\ell (\gamma):= \sum_{v\in \gamma} X_{v}.$$ 
Let $\preceq$ denote the usual partial order on $\Z^2$, i.e., $u\preceq v$ if $u$ is coordinate-wise smaller than $v$. For any two points $u$ and $v$ with $u\preceq v$, the last passage time from $u$ to $v$, denoted $T_{u,v}$ is defined by
$$T_{u,v}:=\max_{\gamma} \ell(\gamma)$$
where the maximum is taken over all up/right paths $\gamma$ from $u$ to $v$. The almost surely unique maximizing path, denoted $\Gamma_{u,v}$, will be called a polymer or a geodesic. 
As already mentioned above, it is a paradigm example of an exactly solvable growth model in the KPZ universality class, and as such has been subject to extensive study. For notational convenience let us denote $T_{\mathbf{1}, \mathbf{n}}$ by $T_{n}$, and the corresponding geodesic by $\Gamma=\Gamma_{n}$. Using connection between exponential LPP and Totally Asymmetric Simple Exclusion Process (TASEP) and the characterization of the invariant measures for the latter, Rost \cite{Ro81} showed that $T_{n}\sim 4n$, and it was shown by Johansson \cite{Jo99} that $n^{-1/3}(T_{n}-4n)$ converges weakly to a scalar multiple of the GUE Tracy-Widom distribution. 

Together with the weight of the polymer, one is also interested in studying its geometry, in particular, how closely the polymer sticks to the straight line joining the endpoints of the path. This is typically measured via the transversal fluctuation of the polymer $\Gamma_{n}$, defined as follows. For $t=0,1,2,\ldots 2n$, let $\Gamma_{n}(t)=(x(t), y(t))$ denote the (unique) vertex of $\Gamma_{n}$ that lies on the anti-diagonal $\{x+y=t\}$. Transversal fluctuation of $\Gamma_{n}$ at $t$, denoted $D_{n}(t)$ is define by
\begin{equation}\label{transversal1}
D_{n}(t):=|x(t)-y(t)|
\end{equation}
and the transversal fluctuation of $\Gamma_{n}$, denoted $D_{n}$ is defined by $D_n:=\max_{t} D_n(t)$. It is well-known \cite{J00, BCS06} that $D_n = n^{2/3+o(1)}$ (see also \cite{BSS14, BGH17, BG18} for more quantitative results), $2/3$ being the characteristic transversal fluctuation exponent for polymers in the KPZ universality class. 
Our focus in this paper is to study how the transversal fluctuation exponent changes in the large deviation regime for $T_{n}$, i.e., when $T_{n}$ is macroscopically smaller or larger compared to the typical value. More precisely, for $\delta>0$, let $\cU_{\delta}=\cU_{\delta}(n)$ denote the upper tail large deviation event
$$\cU_{\delta}(n):=\{T_n\geq (4+\delta)n\}.$$ 
Similarly, for $\delta\in (0,4)$, we define the lower tail large deviation event 
$$\cL_{\delta}(n):=\{T_n\leq (4-\delta)n\}.$$
It is known since the work of Kesten \cite{Kes86} that the large deviation speed for the upper tail is $n$ whereas that of the lower tail is $n^2$. Johansson \cite{Jo99} showed that log-probabilities $\log \P(\cU_{\delta})$ and $\log \P(\cL_{\delta})$ scaled by $n$ and $n^2$ converge to explicit large deviation rate functions.  The upper tail large deviation was also established in \cite{sepLDP}. 

As mentioned above, we are interested in studying the distribution of $D_{n}$ conditional on the large deviation events $\cL_{\delta}$ or $\cU_{\delta}$. In a recent work \cite{BGS17A}, with Allan Sly, we showed that conditional on $\cL_{\delta}$ the geodesic $\Gamma$ is delocalized with high probability, i.e., the transversal fluctuation exponent is $1$. The argument there was based on the fact that the speed of large deviation for the lower tail is $n^2$ and did not use integrability of the model, and hence could also be extended to a large class of other LPP models. For the upper tail large deviation, in the related exactly solvable model of Poissonian LPP \cite{DZ1} showed that the transversal fluctuation in $o(n)$ with high probability conditional on the upper tail large deviation regime, but no finer information was obtained. {However, heuristically one would believe that the optimal way of achieving the upper tail large deviation event is to increase the passage times near the diagonal, and hence the transversal fluctuation exponent should not exceed  $2/3$ conditional on $\cU_{\delta}$}. Our main result in this paper is to identify the transversal fluctuation exponent in the upper tail large deviation regime. We now move towards a precise statement. 

There are two usual ways to define a transversal fluctuation exponent, an upper exponent and a lower exponent which often coincide. For a fixed $\delta>0$, let us define the upper transversal fluctuation exponent $\bar{\xi}=\bar{\xi_{\delta}}$ under the large deviation event $\cU_{\delta}$ as follows: 

\begin{equation}\label{utf}
 \bar{\xi}:=\inf\{\xi'\leq 1: \limsup_{n\to \infty} \P(D_{n}\geq n^{\xi'}\mid \cU_{\delta})=0\}.
 \end{equation}
Similarly we define the lower transversal fluctuation exponent $\underline{\xi}=\underline{\xi_{\delta}}$ under the large deviation event $\cU_{\delta}$ as follows: 
\begin{equation}\label{ltf}
\underline{\xi}:=\sup\{\xi'\geq 0: \liminf_{n\to \infty} \P(D_{n}\geq n^{\xi'}\mid \cU_{\delta})=1\}.
\end{equation}

If $\bar{\xi}=\underline{\xi}$, we say that the transversal fluctuation exponent under $\cU_{\delta}$ exists and denote the common value by $\xi=\xi_{\delta}$. The following main theorem in this paper shows that $\xi_{\delta}$ exists and is independent of $\delta$. 
\begin{figure}[h]
\centering
\includegraphics[scale=.6]{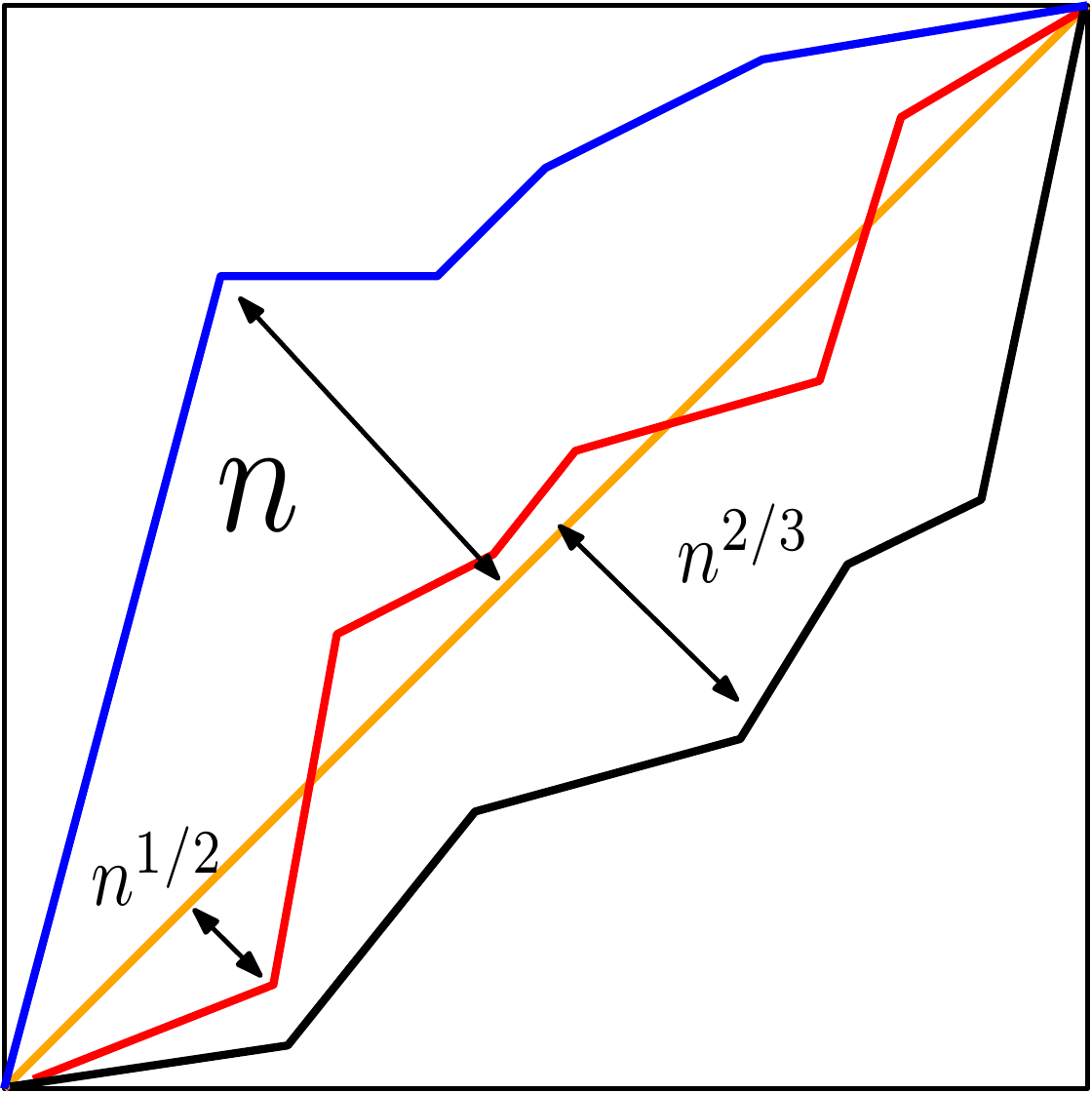}
\caption{The blue, red and black paths denote the polymers in the lower tail, upper tail, and typical conditions. The transversal fluctuation exponents in the three regimes are $1, 1/2$ and $2/3$ respectively. 
}
\end{figure}
\begin{maintheorem}
\label{t:main}
For each $\delta>0$, $\xi_{\delta}$ exists and is equal to $\frac{1}{2}$.
\end{maintheorem}

{A few remarks are in order: 
Observe that it is plausible that the transversal fluctuation in the upper tail large deviation regime should be smaller than $2/3$. One might try to heuristically predict the exponent $1/2$ in the large deviation regime using the KPZ relations between the longitudinal and transversal exponents (see \cite{Cha-KPZ}) and predicting that the longitudinal fluctuation is $O(1)$ in the large deviation regime unlike the typical $O(n^{1/3})$. We will see later that this is indeed the case.

Note that $n^{1/2}$ is the scale where the entropy of paths is maximized, and it turns out that the best strategy to increase $T_{n}$ is to increase weights in that strip (i.e., in the coupling between the typical field and the field conditioned on $\cU_{\delta}$ given by the FKG inequality the weights differ only in the $n^{1/2}$-width strip around the diagonal).  Even though one expects that the exponent $1/2$ is universal for a large class of passage time distribution, our proof will fully exploit the specific integrability properties of exponential LPP. In contrast with the result of \cite{BGS17A}, which did not require exact solvability, and the recent results in \cite{BSS14, BSS17B, BG18, BHS18} which only required the moderate deviation estimates for the last passage time (that are available for a number of other exactly solvable models of last passage percolation), the proof of Theorem \ref{t:main} will be specific to the exponential case. Indeed, we use the correspondence of the last passage time with the largest eigenvalue of a certain random matrix ensemble (LUE) obtained in \cite{Jo99}. We can also carry out the argument for Brownian LPP where there is a similar correspondence to GUE (see Section \ref{fd}); one believes that the calculations might be doable in all the models where there is a determinantal structure. 

The proof of Theorem \ref{t:main} is divided into two parts, separately showing $\bar{\xi}\leq \frac{1}{2}$ and $\underline{\xi}\geq \frac{1}{2}$. As a matter of fact for both the directions we shall prove quantitative estimates stated below.

Let us start with upper bounding $\bar{\xi}$.
\begin{maintheorem}
\label{t:mainu}
There exists a fixed constant $C_0$ such that for any fixed $\delta>0$, 
$$\limsup_{n\to \infty }\P(D_n \geq n^{1/2} (\log n)^{C_0\log \log n}\mid \cU_{\delta}) \to 0.$$

\end{maintheorem}
The above bound is not optimal and we expect $D_{n}$ to be a tight random variable at scale $n^{1/2}.$ Further elaboration on this point and the reason for the $(\log n)^{C_0\log \log n}$ term above is discussed later in the article. 
The statement for the lower bound is more straightforward. 

\begin{maintheorem}
\label{t:mainl}
Fix $\delta>0$. 
$$\limsup_{n\to \infty}\P(D_n \leq  hn^{1/2}\mid \cU_{\delta}) \to 0 $$
as $h\to 0$.
\end{maintheorem}

Clearly Theorem \ref{t:main} follows from the above two results.
\subsection{Large deviation background}
Large deviations for random growth models has been studied classically starting from the work of Kesten \cite{Kes86}. Under general assumptions on the passage time distribution, it possible to show using Kesten's argument that the speed of large deviation in $n$ for the upper tail and $n^2$ for the lower tail. (Kesten's original argument was for first passage percolation for which the tail behaviors are reversed).
For integrable models of last passage percolation explicit large deviation rate functions was derived by Deuschel and Zeitouni \cite{DZ1} and Sepp{\"a}l{\"a}inen \cite{sepLDP} for Poissonian last passage percolation using RSK correspondence/ Young Tableaux combinatorics and connections to Hammersley process respectively. For LPP with exponential and geometric passage times Johansson \cite{Jo99} obtained the large deviation rate functions using connection to generalized permutations, random matrices and orthogonal polynomial ensembles.  Sepp{\"a}l{\"a}inen \cite{sep98} had also obtained the rate function for the upper tail for exponential LPP using coupling to the totally asymmetric exclusion process (TASEP).
The precise result of Johansson establishing the upper tail large deviation principle in exponential LPP is recorded next. 
\begin{theorem}[\cite{Jo99}]
\label{t:ldp}
For $\delta>0$, 
$$\lim_{n\to \infty} \frac{1}{n} \log \P (\cU_{\delta}(n))=-I(\delta)$$
where 
\begin{equation}
\label{e:i}
I(\delta)=-2+(4+\delta)-2\int_{0}^{4} \log (4+\delta-x) \frac{\sqrt{x(4-x)}}{2\pi x}~dx.
\end{equation}
\end{theorem}

One can verify that $I(0)=0$ and further just by differentiating under the integral sign that $I(\delta)$ is a convex function. However notice that $I(\delta)$ is not uniformly strongly convex but $I'(\delta), I''(\delta)$ converge to $1$ and zero respectively as $\delta$ goes to infinity. 
As a matter of fact Johansson obtain a different expression for $I(\delta)$. The expression above was first obtained by Majumdar and Vergassola \cite{MV09}. 

A key ingredient in our proofs is a non-asymptotic quantitative version of Theorem \ref{t:ldp} (see Theorem \ref{t:ldpm=n}) below. 
A significant part of the work in this paper goes into proving this using connections to Random matrix theory and rigidity properties of the spectrum.

\subsection{Key Ingredients}
The proof of Theorem \ref{t:main} combines a number of different ingredients, and as by products of this proof we also get a number of other results that are of independent interest. The starting point of the analysis is the well-known remarkable correspondence between the last passage time in the Exponential LPP model and the largest eigenvalue of a certain complex Wishart matrix. We recall this correspondence below.

\subsubsection{Correspondence to LUE}
Let $X_{M\times N}$ denote an $M\times N$ matrix with standard complex Gaussian entries, where $M\geq N$. Let $W=W_{N\times N}:=X^*X$ denote the complex Wishart matrix, and let $\widehat{\lambda_1}\geq \widehat{\lambda_2} \geq \cdots \geq \widehat{\lambda_{N}}$ denote the eigenvalues of $W$. Recall that $T_{(1,1), (M,N)}$ denotes the last passage time from $(1,1)$ to $(M,N)$. The following fundamental correspondence between the last passage time and the eigenvalues of Wishart matrix was established by Johansson \cite{Jo99}.

\begin{proposition}[\cite{Jo99}]
\label{p:eigen}
In the above notation, we have 
\begin{equation}\label{input}
\widehat{\lambda_1} \stackrel{d}{=} T_{(1,1), (M,N)}.
\end{equation}
\end{proposition}

There is a more general correspondence  between last passage times with other eigenvalues and the Wishart minor process \cite{adler} but we shall not need that. We shall use Proposition \ref{p:eigen} in the following way. The eigenvalues of a complex Wishart matrix form a determinantal point process on $\mathbb{R}$ \cite{Jo99} whose joint density can be explicitly written (see Section \ref{s:wishart}). We shall need a two-fold consequence of this information.

\subsubsection{Sharp Asymptotics of Large Deviation Probability}
\label{s:sharp1}
One can use the so called Coulomb gas techniques \cite{MV09} to obtain the upper tail large deviation rate function for $\widehat{\lambda_1}$ from the joint density of $(\widehat{\lambda_1}, \ldots, \widehat{\lambda_{N}})$. We shall use fine rigidity results available for the eigenvalues \cite{GT14, BYY14} to refine the calculation to obtain sharp asymptotics of log of the large deviations probabilities for $\widehat{\lambda_1}$ up to  sub-polynomial correction terms. 

For our convenience we shall work with scaled eigenvalues. Let ${\lambda_1}\geq {\lambda_2} \geq \cdots \geq {\lambda_{N}}$ denote the ordered eigenvalues of the matrix $\frac{1}{M} X^* X$. The reason for this scaling is that the empirical distribution of eigenvalues $\frac{1}{N}\sum \delta_{\lambda_i}$ converges to Marchenko-Pastur distribution of appropriate parameter (see Section \ref{s:wishart}) making certain calculations more transparent. We shall prove the following two results for large deviation of $\lambda_1$ that are of independent interest.

As we shall have the occasion to use different pairs of $(M,N)$ we shall denote the corresponding probability measures by $\P_{M,N}$. In the case $M=N$, i.e., where $X$ is a square matrix, we shall simply denote the probability by $\P_{N}$. In the square matrix case the log probability can be evaluated up to a constant, and that is our first result in this vein.

\begin{theorem}
\label{t:ldpm=n}
Let $M=N$ and and fix  $\delta_0>0$. Then for each $\delta\in (\delta_0,\infty)$
$$\log \P_{N}(\lambda_1 > 4+\delta)= -NI(\delta)-\log N+O(1)$$
as $N\to \infty$ where the $O(\cdot)$ term just depends on $\delta_0$ and $I(\delta)$ be defined by \eqref{e:i}.  
\end{theorem}

A similar but less precise result is obtained in the case $M\neq N$. Before making a precise statement let us first explain the result. Recall that if $M,N\to \infty$ and $\frac{N}{M}\to y\in (0,1]$, then $\lambda_1 \to (1+\sqrt{y})^2$ almost surely, and one can also show that $\frac{1}{N}\log \P_{M,N}(\lambda_{1}>(1+\sqrt{y})^{2}+\varepsilon)$ converges to a precise large deviation rate function depending on $y$ and $\varepsilon$. However, for our applications, we will need to deal with the situation where $M=N+o(N)$ which is much more delicate.  
For $y\in (0,1]$ and {$\delta>0$}, let us define 

\begin{equation}
\label{e: rate1}
J_{y}(\delta):=\int_{(1-\sqrt{y})^2}^{(1+\sqrt{y})^2} \log (4+\delta-x) \frac{1}{2\pi x y} \sqrt{((1+\sqrt{y})^2-x)(x-(1-\sqrt{y})^2)}~dx;~\text{and}
\end{equation}

\begin{equation}
\label{e: rate2}
I_{y}(\delta):=- (2+y^{-1})+\log y +1+(4+\delta)y^{-1}-(y^{-1}-1)(\log (4+\delta))-2J_{y}(\delta).
\end{equation}
We have the following theorem. 
\begin{theorem}
\label{t:ldpgen}
There exists a universal constant $c>0$ such that for each $M,N$ with  $1.1N\ge M\ge N$, and each $\delta_0,L$ with $0<\delta_0<L$ and each $\delta\in (\delta_0,L)$ we have
$$\log \P_{M,N}(\lambda_1> (4+\delta))=-NI_{y}(\delta)+ O((\log N)^{c\log \log N}),$$ where $y=\frac{N}{M}\in (0,1]$ and the constant in the $O(\cdot)$ term is just a function of $\delta_0$ and $L$.
\end{theorem}
First, as a sanity check observe that for $y=1$, $I_{y}(\delta)=I(\delta)$ as is expected.
Further we would like to point here out that the above theorem and hence Theorem \ref{t:mainu} is not quantitatively optimal. The $(\log N)^{c\log \log N}$ term is an artefact of our proof that comes from using the universal rigidity result of \cite{BYY14} for all eigenvalues.
 Indeed, one expect the error term in Theorem \ref{t:ldpgen} to be also $-\log N+O(1)$ as in Theorem \ref{t:ldpm=n}.  We elaborate more on this point later in the section where the rigidity input Theorem \ref{t:rigidity} is stated. 
\subsubsection{Comparison of Largest Eigenvalues} Many of our arguments would rely on stochastic comparisons of the eigenvalue ensembles of covariance matrices of various dimensions.  Let $X$ be as above  i.e., n $M\times N$ matrix with standard complex Gaussian entries and let $Y$ be an $(M+1)\times (N-1)$ matrix with standard complex Gaussian entries. Let $\widetilde{W}=\widetilde{W}_{(N-1)\times (N-1)}:=Y^*Y$ denote the complex Wishart matrix, and let $\tilde{\lambda}_1\geq \tilde{\lambda}_2 \geq \cdots \geq \tilde{\lambda}_{N-1}$ denote the eigenvalues of $\widetilde{W}$. Observe that we do not scale this matrix. Recall also that $\widehat{\lambda_1} \geq \widehat{\lambda_2}\geq \cdot \geq \widehat{\lambda_{N}}$ denote the eigenvalues of the unscaled matrix $X^*X$. We have the following result.

\begin{theorem}
\label{t:dom}
In the above set-up, assume $M-N$ is even. Then there exists a coupling such that almost surely $$(\tilde{\lambda}_1,\tilde{\lambda}_2, \ldots,\tilde{\lambda}_{N-1})\subset 
(\widehat{{\lambda}_1},\widehat{{\lambda}_2},\ldots ,\widehat{{\lambda}_{N}}).$$
In particular 
we have $\widehat{\lambda_1} \succeq \tilde{\lambda}_{1}$, and equivalently $T_{(1,1),(M,N)}\succeq T_{(1,1),(M+1,N-1)}$ where $\succeq$ denotes stochastic domination. \end{theorem}
The statement holds also in the case $M-N$ odd with minor modifications in the proof which are indicated later.
{A related discussion involving interlacing inequalities and relating polymer weights to minor processes appear in \cite{adler}.} Observe that, Theorem \ref{t:dom} is not a consequence of the standard interlacing results for eigenvalues, as there does not appear to be a natural coupling which has $\widetilde{W}$ as a minor of $W$. Instead our proof invokes an abstract result of Lyons \cite{lyons} about stochastic comparisons of determinantal point processes whose kernels are ordered. We thank Manjunath Krishnapur for showing us how this result can be used to prove Theorem \ref{t:dom}. We reproduce this proof later (See Section \ref{secdom}).

\subsubsection{Outline of the argument}
We now give a brief description of how the argument proceeds once we have the above ingredients at our disposal, together with the correspondence given by Proposition \ref{p:eigen}. Recall that, $D_{n}(t)$ denotes the transversal fluctuation of the geodesic $\Gamma_{n}$ at the anti-diagonal $\{x+y=t\}$. Observe that for the lower bound it suffices to only control the transversal fluctuation at the main anti-diagonal, i.e., on the line $\{x+y=n\}$. For the sake of clarity we shall also restrict our discussion of the upper bound here to this case i.e., upper bounding of $D_{n}(n)$ only. The proof of the more general case is similar. 

The main idea of the both the upper and lower bound is the same, For $v=(v_1,v_2)$ with $v_1+v_2=n$, let $\Gamma^*(v)$ denote the highest weight path from $\mathbf{1}$ to $\mathbf{n}$ that passes through $v$. The basic estimate is to compute, up to  a sufficient degree of accuracy, the quantity:

\begin{equation}
\label{e:ratio}
\dfrac{\P(\ell (\Gamma^*(v))\geq (4+\delta) n)}{\P(\cU_{\delta}(n))}.
\end{equation}

As $\ell(\Gamma^*(v))= T_{\mathbf{1},v}+T_{v,\mathbf{n}}$ (up to an error of $X_{v}$ which we shall show can be ignored), we can use Theorem \ref{t:ldpgen} to obtain an upper bound of the numerator of \eqref{e:ratio}, whereas Theorem \ref{t:ldpm=n} provides a lower bound of the denominator. However to compare the upper and lower bounds, using the connection to eigenvalues,  one needs to compare the large deviation rate functions for the largest eigenvalue of a square and a non-square Wishart Matrix with aspect ration $y$. This is done in Section \ref{s:polyub} which establishes the following quadratic correction:
$$I_{y}(\cdot)=I({\cdot})+\Theta(\frac{c^2}{n}),$$
where $y=\frac{n-c}{n+c}.$
 Combined, these provide a upper bound of $o(\frac{1}{n^2})$ for the \eqref{e:ratio} which is thus  $o(1)$ even after summing over all $v$ with $|v_1-v_2|\ge n^{1/2}(\log n)^{c\log \log n}$. 

For the lower bound, we shall focus on the point $v_*=(\frac{n}{2}, \frac{n}{2})$. In this case, we can estimate both numerator and denominator of \eqref{e:ratio} using Theorem \ref{t:ldpm=n}. This sharp estimates lets us conclude that \eqref{e:ratio} is $O(n^{-1/2})$ for $v=v_{*}$. Using the comparison result Theorem \ref{t:dom}, we can then obtain similar estimates for other $v$s on the main anti-diagonal, and this concludes the proof of Theorem \ref{t:mainl}. Observe that Theorem \ref{t:dom} is crucial here, as the weaker estimate Theorem \ref{t:ldpgen} will not be sufficient for our purposes. 
\subsection{Organization of the paper}
The rest of this paper is organized as follows. In Section \ref{s:wishart} we recall the basic facts about the eigenvalue distribution of a complex Wishart matrix, and some concentration and rigidity results that we need. In Section \ref{s:sharp}, we establish the sharp asymptotics of the large deviation probability for the largest eigenvalue, namely we establish Theorem \ref{t:ldpm=n} and Theorem \ref{t:ldpgen}. Section \ref{secdom} is devoted to proving the domination result Theorem \ref{t:dom}. A key estimate in comparing the rate functions for square and non-square Wishart matrices is obtained in Section \ref{s:polyub}. Sections \ref{ptu} and \ref{ptl} are devoted to the proofs of Theorems \ref{t:mainu} and \ref{t:mainl} respectively. 
Various possible extensions to other models and future research directions are outlined in Section \ref{fd}. Finally technical proofs of certain results are included in the Appendix (Section \ref{calculations}). 
\subsection*{Acknowledgements}
The authors thank Paul Bourgade, Amir Dembo,    Subhroshekhar Ghosh, Manjunath Krishnapur, Satya Majumdar, Allan Sly, HT Yau, Jun Yin and Ofer Zeitouni for several useful discussions. RB is partially supported by an ICTS Simons Junior Faculty Fellowship and a Ramanujan Fellowship from Govt. of India. Part of this research was performed during a visit of RB to UC Berkeley Statistics department, he gratefully acknowledges the hospitality.

\section{Eigenvalues of Wishart matrices}
\label{s:wishart}
In this section we recall the basics about the joint distribution of eigenvalues of complex Wishart matrix that will be used to prove Theorem \ref{t:ldpm=n} and Theorem \ref{t:ldpgen}.
\subsection{Joint density of eigenvalues}
Let us recall the setting in Section \ref{s:sharp1}. Let $M\geq N$ and $X$ denote an $M\times N$ matrix of i.i.d.\ complex Gaussian variables with variance $1$ (that is the real and imaginary parts are independent $N(0,\frac{1}{2})$ variables). Clearly $\frac{1}{M}X^*X$ almost surely has all positive eigenvalues. Let $\lambda_1 \geq \lambda_2 \geq \cdots \geq \lambda_{N}$ denote the eigenvalues of $\frac{1}{M}X^*X$. Let $\Lambda_{N}$ denote the cone
$$\Lambda_{N}:=\{(\lambda_1, \lambda_2, \ldots, \lambda_{N})\in \R^N: \lambda_1 \geq \lambda_2 \geq \cdots \geq \lambda_{N}\}.$$

It is a well-known fact (see \cite{Jo99,BBP}) that for $\underline{\lambda}=(\lambda_{1}, \ldots , \lambda_{N})\in \Lambda_{N}$ the joint eigenvalue density of the scaled Wishart matrix is given by

\begin{equation}
\label{e:density}
f(\underline{\lambda})=f_{M,N}(\underline{\lambda})= \frac{1}{Z_{M,N}}V(\underline{\lambda})^2\prod_{i=1}^{N}\lambda_i^{M-N}e^{-M\sum_{i=1}^N\lambda_i}, 
\end{equation}
 where
$$V(\underline{\lambda}):= \prod_{i<j} (\lambda_i-\lambda_j),$$
and the partition function $Z_{M,N}$ is given by 

\begin{equation}
\label{e:partition}
Z_{M,N}=\frac{\prod_{j=0}^{N-1}j!(M-N+j)!}{M^{NM}}. 
\end{equation}
For convenience of notation, we shall denote the above density by $f_{N}$ in the case $M=N$.

\subsection{Marchenko-Pastur law, and large deviation for the leading eigenvalue}
\label{s:heuristic}
Using \eqref{e:density}, one can show that (see \cite{MP,GT14}) the empirical spectral measure $\frac{1}{N}\sum_{i=1}^{N} \delta_{\lambda_i}$ of the matrix $\frac{1}{M} XX^*$ converges (as $M\to \infty$ and $\frac{N}{M}\to y\in (0,1]$) to the Marchenko-Pastur law $\mathsf{MP}_{y}$ with parameter $y$ with the density 
\begin{equation}
\label{e:mpy}
d\mathsf{MP}_{y}(x)= \frac{1}{2\pi xy} \sqrt{(b-x)(x-a)}~dx; \qquad x\in (a,b)
\end{equation}
where $a=(1-\sqrt{y})^2$ and $b=(1+\sqrt{y})^2$.
Let us also record the particular case $y=1$ of the standard Marchenko-Pastor law $\mathsf{MP}$ separately for convenience.

\begin{equation}
\label{e:mp1}
d\mathsf{MP}(x)= \frac{1}{2\pi x} \sqrt{x(4-x)}~dx; \qquad x\in (0,4).
\end{equation}

Using \eqref{e:mp1} we can now sketch a quick proof of 
\begin{equation}
\label{e:ldpeigen}
\lim_{N\to \infty} \log \frac{\P_{N}(\lambda_1\geq (4+\delta))}{N}=-I(\delta)
\end{equation}
using the so called Coulomb gas methods. Together with Proposition \ref{p:eigen}, the above  immediately proves Theorem \ref{t:ldp} . Our arguments will require a quantitative refinement of the same. Let $\underline{\lambda}^{(1)}= (\lambda_2, \ldots , \lambda_{N})$ and $V(\lambda_1; \underline{\lambda}^{(1)}):=\prod_{j\neq 1}(\lambda_1-\lambda_j)$. Using \eqref{e:density} we write down $\P(\lambda_1\geq (4+\delta))$ below.
\begin{eqnarray*}
\P(\lambda_1 \geq (4+\delta)) &=&\int_{\lambda_1\ge (4+\delta)}f_{N}(\underline{\lambda}) d\underline{\lambda}\\
&=& \frac{Z_{N-1,N-1}}{Z_{N,N}}\int_{\lambda_1\ge (4+\delta)} e^{-N\lambda_1}\left( \int_{\underline{\lambda}^{(1)}:\lambda_2\leq \lambda_1} V(\lambda_1; \underline{\lambda}^{(1)})^2 e^{-\sum_{i=2}^N\lambda_i}f_{N-1,N-1} d\underline{\lambda}^{(1)}\right) d\lambda_1\\
&=& \frac{Z_{N-1,N-1}}{Z_{N,N}}\int_{\lambda_1\ge (4+\delta)} \int_{\underline{\lambda}^{(1)}:\lambda_2\leq \lambda_1} \exp \left(N (-\lambda_1 + \frac{2\log V(\lambda_1; \underline{\lambda}^{(1)})}{N} + \frac{1}{N} \sum_{2}^{N} \lambda_i)  \right) d\underline{\lambda}^{(1)}d\lambda_1. 
\end{eqnarray*}

Now, using the fact the empirical spectral measure is close to $\mathsf{MP}$ it follows that for large $N$, the inside integral is

$$\approx \exp\left (N (-\lambda_1+ 2\int \log (\lambda_1-x)d\mathsf{MP}(x)+ \int x d\mathsf{MP}(x)) \right),$$ where $\approx$ is used to denote approximately whose meaning we will not be making precise since this is supposed to only be a sketch of the argument.
Observe that the term inside the bracket is maximized at $\lambda_1=(4+\delta)$ with exponential decay beyond that, and hence it follows that 

$$\log \P(\lambda_1 \geq (4+\delta)) \approx \log \frac{Z_{N-1,N-1}}{Z_{N,N}}+ N\left(-(4+\delta) + 2\int \log (4+\delta-x)d\mathsf{MP}(x)+ \int x d\mathsf{MP}(x)\right).$$
Observing from \eqref{e:partition} that 
\begin{equation}
\label{e:partitionf}
\log \frac{Z_{N-1,N-1}}{Z_{N,N}}= 3N+O(1)
\end{equation}
and from \eqref{e:mpy} that 
\begin{equation}
\label{e:mpmean}
\int x d\mathsf{MP}_y(x)=1
\end{equation}
for all $y\in (0,1]$, \eqref{e:ldpeigen} follows. To make the above argument rigorous, one needs certain estimates of how close the empirical spectral distribution is to the Marchenko-Pastur law. We shall need a quantitative variant of the above which also works for rectangular Wishart matrices, i.e., when $M>N$, as long as $M-N=o(N)$.  Furthermore, the proof of Theorem \ref{t:mainl} will depend on certain precise polynomial correction given by the extra log factor in the exponent in Theorem \ref{t:ldpm=n}.  Below we record a number of known facts about the empirical spectral distribution of complex Wishart matrices that we will use later. 

\subsection{Eigenvalue Rigidity for Wishart Matrices and its consequences}
First we need a result to show that linear statistic of the eigenvalues are concentrated around their expectation under certain assumption. To this end we have the following sub-Gaussian concentration result for Lipschitz functionals from \cite{GZ00} which relies on the Gaussian Log-sobolev inequality. Recall the basic set-up. Let $W=\frac{1}{M}
X^*X$ be a (scaled) complex Wishart matrix with eigenvalues $\lambda_1>\cdots > \lambda_{N}$ where $X$ is an $M\times N$ matrix of i.i.d.\ standard complex Gaussian entries. 
For a function $f:\R\to \R$ let us define 
$$\tr(f):=\frac{1}{N}\sum f(\lambda_i).$$

\begin{theorem}[\cite{GZ00}]
\label{t:conc} 
For any Lipschitz $f$, there exists $C>0$ depending on the Lipschitz constant of $f$ such that for all $M,N$ and all $\delta>0$ we have
$$\P\left(|\tr(f)-\E(\tr(f))|\ge \delta\frac{M+N}{N}\right)\le e^{-C\delta^2(M+N)^2}.$$
\end{theorem}

Observe that Theorem \ref{t:conc} provides concentration of the empirical spectral measure around its mean. However for our purposes we need to approximate the expected empirical spectral measure with the ``limiting spectral measure"  $d\mathsf{MP}_{y}$ where $y=\frac{N}{M}$ (this is an abuse of terminology since $y$ in $N$ dependent) in a quantitative manner. To this end we record the following two results. 

\begin{theorem}[\cite{GT14}]
\label{t:gotzeks}
Let $M=N$ and let $\mathsf{ESM}$ denote the expected empirical spectral distribution of $W$. There exists an absolute constant $C$ such that 
$d_{{\rm KS}}(\mathsf{ESM}, \mathsf{MP})\leq CN^{-1}$ for all $N$ where $d_{{\rm KS}}(\cdot, \cdot)$ denote the Kolmogorov-Smirnov distance between two distributions. 
\end{theorem}

We shall use this Theorem to show that for  specific choices of sufficiently nice functions $f$, we have 
$$\left |  \int f ~d\mathsf{ESM} - \int f~ d\mathsf{MP} \right |=O(N^{-1}).$$
Note that this would follow immediately, using integration by parts, for functions with $||f'||_1<\infty$. Unfortunately, in our case the function $f$ will be almost linear and hence we cannot apply that directly. We shall need certain tail estimates on $\mathsf{ESM}$ to get around that issue. Towards this we first record the following standard tail estimate. 

\begin{lemma}
\label{standconc}
There exists $c>0$ such that for any $L>4.5$, and $N,$ $$\P(\lambda_1>L)\le e^{-cLN}.$$
\end{lemma}

We shall omit the standard proof of the above fact. Observe however that for $L$ sufficiently large, one can prove this by appealing to Proposition \ref{p:eigen} and taking a union bound over all paths. The general case can be proved by appealing to concentration of measure results such as in \cite{talagrand}. A more refined version of this also proved using sub-additivity later in Proposition \ref{t:diaguniform}. The following lemma is an almost immediate consequence of the above result. 
\begin{lemma}
\label{tail}
There exists a constant $C$ such that $$\mathsf{ESM}[x,\infty]\le e^{-CN x} $$
for all $x\ge 4.5$ and $N$. 
\end{lemma}

\begin{proof} The proof follows from Lemma \ref{standconc} and the bound that $\mathsf{ESM}[x, \infty]\le N\P(\lambda_1\ge x).$
\end{proof}

Our next result uses the previous two lemmas to control $\left |  \int f ~d\mathsf{ESM} - \int f~ d\mathsf{MP} \right |$.

\begin{proposition}
\label{l:diffcontrol}
Let $f$ be a continuous and piecewise $C^1$ function on $[0,\infty)$ such that $||f'||_{\infty}\leq \infty$. Then there exists absolute constants $C_{3}$ and $C_{4}$ independent of $f$ such that
$$\left |  \int f ~d\mathsf{ESM} - \int f~ d\mathsf{MP} \right |\le \frac{C_{3}A_{f}}{N}+ B_{f} e^{-C_4N}$$
where $A_{f}=\int_{0}^{5} |f'(x)|~dx$ and $B_f=||f'||_{\infty}$.
\end{proposition}

\begin{proof}
Write $f=f_1+f_2$ where $f_1$ is a continuous function which agrees with $f$ on $[0,5]$ and is constant on $[5,\infty].$
Thus 
\begin{align*}
\left |  \int f ~d\mathsf{ESM} - \int f~ d\mathsf{MP} \right |\le \left |  \int f_1 ~d\mathsf{ESM} - \int f_1~ d\mathsf{MP} \right |+\left |  \int f_2 ~d\mathsf{ESM} - \int f_2~ d\mathsf{MP} \right |,
\end{align*} 
Now note that $\int f_2~ d\mathsf{MP}=0$ since their supports are disjoint. Thus the result follows from the following two lemmas which bound each of the above terms.
\end{proof}

\begin{lem}$\left |  \int f_1 ~d\mathsf{ESM} - \int f_1~ d\mathsf{MP} \right |=O(\frac{A_{f}}{N}).$
\end{lem}

\begin{proof}
By integration by parts \begin{equation}
\label{e:ksapp}
\left |  \int f_1 ~d\mathsf{ESM} - \int f_1~ d\mathsf{MP} \right|\le \left |  \int |f'_1| ~\left|\mathsf{ESM}(x) -\mathsf{MP}(x)\right| \right|=O\left(\frac{1}{n}\right)\int_{0}^5 |f_1'|.
\end{equation}
Here we ignore the boundary terms because $\mathsf{ESM}(x)\to 1$ as $x\to \infty$ and $\mathsf{MP}(x)=1$ for all $x\ge 4.$ and $f_1$ is bounded. The $1/n$ term is then by Theorem \ref{t:gotzeks}.
\end{proof}
\begin{lem}$\int_{5}^{\infty} f_2 ~d\mathsf{ESM}=O(B_f e^{-CN}).$
\end{lem}
\begin{proof}Since translating the measure $\mathsf{ESM}$ by $1$ does not change the above integral, by integration by parts, the above is  $\int_{5}^{\infty}f_2' (1-\mathsf{ESM}([0,x])dx$ .
The first boundary term is $f_2(5)(1-\mathsf{ESM}([0,5])=0;$ the other boundary term vanishes as well since $f_2$ grows at most linearly while $(1-\mathsf{ESM}([0,x])$ decays exponentially by Lemma \ref{tail}.
\end{proof}

Observe that together with Theorem \ref{t:conc} this will imply (for sufficiently nice $f$)
\begin{equation}
\label{e:o1}
\sum_{i=1}^{N} f(\lambda_i)- N\int f~ d\mathsf{MP}= O(1)
\end{equation}
with high probability. This is a quantitative variant of the well-known CLT smooth linear statistics for Wishart matrices \cite{Bai}. For our problem we require the following variant as well  for $M\neq N$ case  where the centering term is now changed to $\int f~ d\mathsf{MP}_{y}$ where $y=N/M$.

\begin{lem}\label{approxsm}Let $f$ be a continuous and piecewise $C^1$ function on $[0,\infty)$ such that $||f'||_{\infty}\leq \infty$ and $M=N+o(N)<1.1 N.$ Then there exists constants $C_{3}$ and $C_{4}$ depending only on the Lipschitz constant of $f$ such that
$$\left |  \int f ~d\mathsf{ESM} - \int f~ d\mathsf{MP}_y \right |=\frac{C_3g(N)+C_4}{N},$$
\end{lem}
where throughout the sequel for brevity we will denote the term $(\log N)^{c\log \log N}$  appearing in Theorem \ref{t:rigidity} by $g(N).$
To prove the above, we rely on the following rigidity result about the locations of the eigenvalues to their classical locations.  
For $j=1,2,\ldots, N$ let $\gamma_j=\gamma_{j,M,N}$ denote the classical location of the eigenvalues of $\frac{1}{M}XX^*$, i.e., $\gamma_{j,M,N}$ are the solutions of the equations 

$$ \int_{(1-\sqrt{y})^2}^{\gamma_{j,M,N}} d\mathsf{MP}_{y}(x)= 1-\frac{j}{N}$$
where $y=\frac{M}{N}$. The following theorem which is as extension of Lemma 5.1 of \cite{BYY14}  for the case of non-square Wishart matrices and has the same proof as the latter was communicated to us by Paul Bourgade. 
This is a rigidity result which gives comparison between the classical locations $\gamma_j$ and $\lambda_j$. 

\begin{theorem}\cite[Lemma 5.1]{BYY14} 
\label{t:rigidity}
Let $M-N=o(N)$. For $c>0$, let  $\ce_{c}$ denote the event that 
$$\left\{\exists j\in [(\log N)^{c\log \log N}, N-(\log N)^{c\log \log N}] \text{ such that } |\lambda_j -\gamma_j|\ge  \frac{c(\log N)^{c\log \log N}}{\min (j,N+1-j)^{\frac{1}{3}}N^{\frac{2}{3}}}\right\}.$$
There exists $c>0$ such that for all sufficiently large $N$
$$\P(\ce_{c})\leq  e^{-(\log N)^{c\log \log N}}.$$
\end{theorem}
Theorem \ref{t:rigidity} can be used to bound the fluctuation of linear eigenvalue statistic for general Wishart ensembles, as the following lemma demonstrates.  \begin{lemma}
\label{l:approx}
 Let $f:[0,\infty)\to \R$ be a  Lipschitz function. Let $4+\kappa > \lambda_1 \geq \cdots \geq \lambda_{N}>0$ be such that $\underline{\lambda}$ satisfies the event $\ce^{c}_{c}$ of Theorem \ref{t:rigidity} as well as $\lambda_1<5$. Then we have 
\begin{equation}
\label{e:o2}
\sum_{i=1}^{N} f(\lambda_i)- N\int f~ d\mathsf{MP}_{y}= O(g(N))
\end{equation}
where the constant in $O(\cdot)$ depends on the Lipschitz constant of $f$.
\end{lemma}
\begin{proof}
It is clear that from definition of $\ce_{c}$ that if $f$ is Lipschitz we have  
$$ \sum_{i=1}^{N} |f(\lambda_i)-f(\gamma_i)|= g(N) O\left(\frac{1}{N}\sum_{i=1}^{N} \frac{1}{\min (\frac{j}{N},\frac{N+1-j}{N})^{\frac{1}{3}}}\right).$$
By observing that $x^{-1/3}$ is integrable at $0$, it follows that the term in the bracket on the RHS of the above display is $O(1)$ and hence it suffices to show that 
$$\sum_{i=1}^{N} f(\gamma_i)- N\int f~ d\mathsf{MP}_{y}=O(1).$$
To this end, observe that 
$$ |\frac{1}{N} f(\gamma_{i})- \int_{\gamma_{i-1}}^{\gamma_i} f(x)~  d\mathsf{MP}_{y}(x)|\leq \int_{\gamma_{i-1}}^{\gamma_i} \int_{x}^{\gamma_i} |f'(z)|~dz~ d\mathsf{MP}_{y}(x) \leq \frac{1}{N}\int_{\gamma_{i-1}}^{\gamma_i} |f'(z)|~dz.$$
Summing over all $i$, and using that $f'$ is bounded gives the desired result.
\end{proof}
Given the above ingredients we can now finish the proof of Lemma \ref{approxsm}.
\begin{proof}[Proof of Lemma \ref{approxsm}]
By  Theorem \ref{t:conc}, there exists a constant $C$ depending on the Lipschitz constant of $f$ such that with probability at least $\frac{1}{2},$
\begin{equation}\label{e:o10}
|\sum_{i=1}^{N} f(\lambda_i)- N\int f ~d\mathsf{ESM}|\le C 
\end{equation}
Thus for large enough $N,$ \eqref{e:o2}, implies that with probability at least $1/4,$  both \eqref{e:o2} and \eqref{e:o10} hold allowing us to conclude that $|\int f ~d\mathsf{ESM}-\int f~ d\mathsf{MP}_{y}|=O(g(N)+C)=O(g(N)).$
\end{proof}

Note that even though one expects logarithmic terms in the rigidity estimates for the location of the eigenvalues, for smooth linear statistics like on the LHS of \eqref{e:o2}, one expects cancellations to occur and the RHS of the same to be $O(1).$ Such results indeed appear in the literature (see for example \cite[Theorem 1.1]{Bai} which was pointed to us by Ofer Zeitouni). However since in this article we only aim to establish the fluctuation exponent to be $1/2$ according to the definitions \eqref{utf} and \eqref{ltf}, and the form of Theorem \ref{t:rigidity} helps us obtain some quantitative estimates, we choose to work with it. In future work we aim to investigate refined questions such as existence of scaling limits where one would need to rely on estimates such as the ones appearing in \cite{Bai}. A more elaborate discussion on such questions is presented in Section \ref{fd}. 

\section{Sharp Asymptotics of the Large Deviation Probabilities}
\label{s:sharp}
In this section we shall prove the key results Theorems \ref{t:ldpm=n} and \ref{t:ldpgen} using the eigenvalue rigidity results for Wishart matrices .
Recall the basic set-up of previous sections, i.e., $\lambda_1$ is the largest eigenvalue of $\frac{1}{M}XX^{*}$ where $X$ is an $N\times M$ complex Gaussian matrix of i.i.d.\ entries. 
Furthermore for our purposes we record the following weaker but uniform over $\delta$ version of the above theorem which is just a consequence of subadditivity. 

\begin{prp}\label{t:diaguniform}
Let $M=N$  and $I(\delta)$ be defined by \eqref{e:i}. Then for all $\delta>0$
$$\log \P(\lambda_1 > 4+\delta)\le -NI(\delta)+O(4+\delta).$$
\end{prp}

\begin{proof} We provide the quick proof here.  We start by observing that  
$a_n:=\log \P(\lambda_1 > 4+\delta)$ as a function of $n$ is almost but unfortunately not quite a sub-additive sequence. Recalling \eqref{input}, we will consider $T_n$ instead. Now the reason for the failure of the above sequence to be sub-additive is 
that the relation $T_{\mathbf{1}, \mathbf{n}}+T_{\mathbf{n}, \mathbf{n+m}}\le T_{\mathbf{1},\mathbf{n+m}}$ is not deterministically true.
Note that this is because the LHS counts the variable $X_{n,n}$ twice.  Thus to make things sub-additive let $T'_{u,v}=T_{u,v}-X_{v},$ i.e. the last entry of the path is omitted.  As before for brevity let $T'_n=T'_{\mathbf{1},\mathbf{n}}.$
Then it is easy to see that $\P(T'_n\ge (4+\delta)n)$ is sub-additive. Furthermore, the rate functions of $T'$ and $T$ are the same, i.e., for any $\delta>0,$
$$
\lim \frac{a_n}{n}=\lim \frac{a'_n}{n}=-I(\delta),
$$
where $a'_n:=\log \P(T'_n > (4+\delta)n).$ But the sequence $a'_n$ is now sub-additive and hence for every $\delta,$ and any $n,$ we have $a'_n\le -nI(\delta).$ However note that we want to bound $a_n$ and not $a'_n.$ We however notice that $T_{n}\preceq T'_{n+1}$ and hence 
\begin{align*}
\P(T_{n}\ge (4+\delta)n)\le \P(T'_{n+1}\ge (4+\delta)n)\le -(n+1)I(\delta') 
\end{align*}
where $(4+\delta')(n+1)=(4+\delta)n.$
Thus $\delta'=\frac{\delta n-4}{n+1}=\delta-\frac{4+\delta}{n+1}.$
Now one can observe from \eqref{e:i} that $I'(\delta)\le 2$ for all $\delta.$
Thus $I(\delta')\ge I(\delta)-O(\frac{4+\delta}{n+1}).$
Hence $a_{n}\le -nI(\delta)-I(\delta)+O(4+\delta)\le -nI(\delta)+O(4+\delta).$
\end{proof}

For the case $M\neq N$, we shall use Theorem \ref{t:rigidity} instead of Theorem \ref{t:gotzeks} and hence get the stated weaker estimate. 
As already mentioned above the proofs of Theorems \ref{t:ldpm=n} and \ref{t:ldpgen} are largely identical, so we shall, for the most part, concentrate on the more general case of Theorem \ref{t:ldpgen}, indicating at the end how the sharper result of Theorem \ref{t:ldpm=n} can be obtained for the case $M=N$.

We need a series of lemmas to prove Theorem \ref{t:ldpgen}. We start with obtaining a sharp growth rate for the partition function. 

\begin{lemma}
\label{l:partition}
Let $M=N+o(N)$ and $Z_{M,N}$ be defined by \eqref{e:partition}. Then we have 
\begin{align*}
\log \frac{Z_{M-1,N-1}}{Z_{M,N}}&= 2N+M-N\log \frac{N}{M}+ O(1),\\
&= N(2+y^{-1})-N\log y+ O(1),
\end{align*}
\end{lemma}

\begin{proof} Recalling \eqref{e:partition} we observe that $$\frac{Z_{M,N}}{Z_{M-1,N-1}}=\frac{(N-1)!(M-1)!(M-1)^{(N-1)(M-1)}}{M^{NM}}$$ 
Simplifying using Stirling's approximation:
\begin{align*}
\log\frac{Z_{M,N}}{Z_{M-1,N-1}}=&O(1)+(-N-M)+(N-1/2)\log(N-1)+
(M-1/2)\log(M-1)\\
&+(N-1)(M-1)\log(M-1)-NM\log(M),\\
&=-2N-M+ (N\log N+M\log M -(N+M)\log M)+O(1).
\end{align*}
\end{proof}

The next lemma shows that conditional on $\lambda_1>(4+\delta)$, it is exponentially unlikely that 
$\lambda_2 > (4+\delta/2)$. We will not use it explicitly but since our methods yield a soft proof of this and we could not locate such a statement in the literature we record it for future use. For brevity we state the result only in the case $M=N$ (the proof for the general case is similar).
\begin{lemma}
\label{l:toptwo}
There exists a constant $h:=h(\delta)>0$ such that we have  
$$-\log \P(\lambda_1 > (4+\delta), \lambda_2 > (4+\delta/2)) > N(I(\delta)+h(\delta)).$$
\end{lemma}
The proof uses the same Coulomb gas methods as in the proof of Theorem \ref{t:ldpm=n} and hence the arguments for the lemma will be provided after the latter is presented.

We now proceed towards the proof of Theorems \ref{t:ldpm=n} and \ref{t:ldpgen}.
The next lemma shall make precise the heuristic approximation described in Section \ref{s:heuristic}. Recall that  $\underline{\lambda}^{(1)}= (\lambda_2, \ldots , \lambda_{N})$ and $V(\lambda_1; \underline{\lambda}^{(1)}):=\prod_{j\neq 1}(\lambda_1-\lambda_j)$ and  the density $f_{M,N}$ from \eqref{e:density}.
Fix $\delta>0$ and $L>(4+\delta)$. For $\lambda_1\in [4+\delta,L)$, let $g_{\lambda_1}:[0,\infty)\to \R$ be a piece-wise smooth function such that 
$$g_{\lambda_1}(x)=2\log (\lambda_1-x)-x$$
on $[0,4+\delta/2]$ and $g_{\lambda_1}(x)\ge 2\log (\lambda_1-x)-x$ for each $x<\lambda_1$ and $g(x)$ is uniformly  bounded in absolute value on the entire interval $[0,\infty)$ by a constant which is just a function of $\delta$ and $L$. Furthermore, it is easy to choose $g=g_{\lambda_1}$ such that,  
\begin{equation}\label{lipconst1}
\sup_{x\in[0,\infty)}|g'(x)|<C, \text{ and } \int_{0}^5 |g'(x)|dx <C,
\end{equation}
 for some some $C=C(\delta)$ independent of the choice of $\lambda_1\in [4+\delta,\infty).$  
For concreteness we take  
\begin{equation}\label{choice}
g(x)=\left\{\begin{array}{cc}
2\log (\lambda_1-x)-x &\text{ for } x\le 4+\frac{\delta}{2}\\
g(4+\frac{\delta}{2}),& \text{otherwise}.
\end{array}\right.
\end{equation}
 Thus \eqref{lipconst1} holds since $|g'(x)|<1+\frac{2}{\delta}$ for all $x.$   Moreover, because $2\log (\lambda_1-x)-x$ is decreasing for $x<\lambda_1,$ we have $g(x)\ge 2\log (\lambda_1-x)-x$ for all $x\le \lambda_1.$ Thus it follows that for $\lambda_1\in (4+\delta,L)$ we have  
\begin{equation}
\label{e:step1}
\int_{\underline{\lambda}^{(1)}: \lambda_2 \leq \lambda_1} V(\lambda_1; \underline{\lambda}^{(1)})^2 e^{-\sum_{i=2}^N\lambda_i}f_{M-1,N-1} d\underline{\lambda}^{(1)} \leq \int_{\underline{\lambda}^{(1)}} e^{\sum_{i=2}^N g_{\lambda_1}(\lambda_i)}f_{M-1,N-1} d\underline{\lambda}^{(1)} .
\end{equation}
Clearly, we also have 
\begin{equation}
\label{e:step2}
\int_{\underline{\lambda}^{(1)}: \lambda_2 \leq \lambda_1} V(\lambda_1; \underline{\lambda}^{(1)})^2 e^{-\sum_{i=2}^N\lambda_i}f_{M-1,N-1} d\underline{\lambda}^{(1)} \geq \int_{\underline{\lambda}^{(1)}: \lambda_2 \leq 4+\delta/2} e^{\sum_{i=2}^N g_{\lambda_1}(\lambda_i)}f_{M-1,N-1} d\underline{\lambda}^{(1)} 
\end{equation}

The next lemma gives a bound on the RHS of \eqref{e:step1}.
\begin{lemma}
\label{l:intapp}
Let $\delta>0$ and $L>(4+\delta)$. Then uniformly in $\lambda_1\in [(4+\delta),L]$ we have 
\begin{align*}\log \left(  \int_{\underline{\lambda}^{(1)}} e^{\sum_{i=2}^N g_{\lambda_1}(\lambda_i)}f_{M-1,N-1} d\underline{\lambda}^{(1)}\right) &=2N \int \log (\lambda_1-x)d\mathsf{MP}_{y}(x)-N+O(\log N^{c\log \log N}),\\
\log\left(\int_{\underline{\lambda}^{(1)}: \lambda_2 \leq 4+\delta/2} e^{\sum_{i=2}^N g_{\lambda_1}(\lambda_i)}f_{M-1,N-1} d\underline{\lambda}^{(1)} \right)&=2N \int \log (\lambda_1-x)d\mathsf{MP}_{y}(x)-N+O(\log N^{c\log \log N}).
\end{align*}
where $\log N^{c\log \log N}$ is the term appearing in Theorem \ref{t:rigidity}.
Further, if $M=N$, then the error terms in the above display can be replaced by $O(1)$.
\end{lemma}

\begin{proof}
Let $\E_{\mathsf{ESM}}:=\E_{\mathsf{ESM},{M-1,N-1}}$ denote the expectation with respect to the empirical spectral measure of an $(N-1)\times (M-1)$ Wishart matrix followed by an average over the matrix i.e., if $\tilde \lambda_1\ge \tilde \lambda_2\ge \ldots \tilde \lambda_{N-1}$ be the eigenvalues of such a matrix, then for any $f$ $$\E_{\mathsf{ESM}}(f):=\frac{1}{N-2}\E(\sum_{i=1}^{N-1}f(\tilde \lambda_i))= \E(\tr(f)).$$ 
 By Theorem \ref{t:conc} it follows that uniformly over all $\lambda_1 \in [4+\delta, L]$  under the measure $f_{M-1,N-1} d\underline{\lambda}^{(1)},$
$$\P(e^{\sum_{i=2}^N g_{\lambda_1}(\lambda_i)-(N-1)\E_{\mathsf{ESM}}( g_{\lambda_1})}\ge e^{(N-1)\eta})\le  e^{-C\eta^2 N^2}$$ for some $C=C(\delta).$ Note that here we use crucially the uniform bound on the derivative stated in \eqref{lipconst1} required by the hypothesis of Theorem \ref{t:conc}.
Rephrasing we get that for any $y>1,$
$$\P(e^{\sum_{i=2}^N g_{\lambda_1}(\lambda_i)-(N-1)\E_{\mathsf{ESM}} (g_{\lambda_1})}\ge y)\le  e^{-C\log^2 (y)}.$$  
Integrating over $y>1$ we get that, 
$$\int_{\underline{\lambda}^{(1)}} e^{\sum_{i=2}^N g_{\lambda_1}(\lambda_i)-(N-1)\E_{\mathsf{ESM}} (g_{\lambda_1})}f_{M-1,N-1} d\underline{\lambda}^{(1)}=O(1).$$
Note that for the above choice of $g$ in \eqref{choice}, for any value of $\lambda_1>4+\delta,$ 
\begin{equation}\label{unibound543}
\int_{0}^5|g'(x)|\le C(\delta_0).
\end{equation}
Thus by Lemma \ref{approxsm}, and the uniformity assumptions on $g,$ (\eqref{lipconst1})  that, 
\begin{equation}\label{approx12}
|\E_{\mathsf{ESM}}g_{\lambda_1}-\int g_{\lambda_1}d\mathsf{MP}_{y}|= O\left(\frac{\log N^{c\log \log N}}{N}\right).
\end{equation}

Furthermore, for the case $M=N$, the sharper result \eqref{e:o1} can be used to replace the RHS in \eqref{approx12} by $O(\frac{1}{N}).$
Now observe also that by \eqref{e:mpmean}, and the fact that the support of $\mathsf{MP}_{y}$ is contained in $[0,4]$ it follows that 

\begin{equation}\label{approx13}\int g_{\lambda_1}d\mathsf{MP}_{y}(x)=2 \int \log (\lambda_1-x)d\mathsf{MP}_{y}(x)-1.
\end{equation}
Thus using \eqref{e:step1} we obtain that 
\begin{align}\label{ub1}
\int_{\underline{\lambda}^{(1)}: \lambda_2 \leq \lambda_1} V(\lambda_1; \underline{\lambda}^{(1)})^2 e^{-\sum_{i=2}^N\lambda_i}f_{M-1,N-1} d\underline{\lambda}^{(1)} &\leq \int_{\underline{\lambda}^{(1)}} e^{\sum_{i=2}^N g_{\lambda_1}(\lambda_i)}f_{M-1,N-1} d\underline{\lambda}^{(1)} \\
\nonumber
\leq e^{-  N\left[2\int \log (\lambda_1-x)d\mathsf{MP}_{y}(x)-1\right]+O\left({\log N^{c\log \log N}}\right)}.
\end{align}
Similarly for the lower bound using \eqref{e:step2} we get, 
\begin{align}
\label{lb1}
\int_{\underline{\lambda}^{(1)}: \lambda_2 \leq \lambda_1} V(\lambda_1; \underline{\lambda}^{(1)})^2 e^{-\sum_{i=2}^N\lambda_i}f_{M-1,N-1} d\underline{\lambda}^{(1)} &\geq \int_{\underline{\lambda}^{(1)}: \lambda_2 \leq 4+\delta/2} e^{\sum_{i=2}^N g_{\lambda_1}(\lambda_i)}f_{M-1,N-1} d\underline{\lambda}^{(1)}\\
\nonumber
&\geq e^{-  N\left[2\int \log (\lambda_1-x)d\mathsf{MP}_{y}(x)-1\right]+O\left({\log N^{c\log \log N}}\right)},
\end{align}
where for the last inequality we first observe that Theorem \ref{t:conc} implies that
\begin{align*}
\P_{M-1,N-1}\left[|\sum_{i=2}^{N}g_{\lambda_1}(\lambda_i)-(N-1)\E_{\mathsf{ESM}}(g_{\lambda_1})|<A, \mathbf{1}(\lambda_2< (4+\frac{\delta}{2}))\right]&\ge \P\left[|\sum_{i=2}^{N}g_{\lambda_1}(\lambda_i)-(N-1)\E_{\mathsf{ESM}}(g_{\lambda_1})|<A\right],\\
& -\P\left[\mathbf{1}(\lambda_2< (4+\frac{\delta}{2}))\right],\\
&\ge 1/2-e^{-\Theta(n)},
\end{align*}
where the last inequality follows by taking $A=O(1)$ large enough depending on $C$ appearing in Theorem \ref{t:conc}.
The above bound implies that 
\begin{align}\label{lb123}
\int_{\underline{\lambda}^{(1)}: \lambda_2 \leq 4+\delta/2}& e^{\sum_{i=2}^N g_{\lambda_1}(\lambda_i)}f_{M-1,N-1} d\underline{\lambda}^{(1)}\\
\nonumber
&\ge [e^{-(N-1)\E_{\mathsf{ESM}}(g_{\lambda_1})-A}]\P\left[|\sum_{i=2}^{N}g_{\lambda_1}(\lambda_i)-(N-1)\E_{\mathsf{ESM}}(g_{\lambda_1})|<A, \mathbf{1}(\lambda_2< (4+\frac{\delta}{2}))\right],\\
\nonumber
&\ge \frac{e^{-(N-1)\E_{\mathsf{ESM}}(g_{\lambda_1})-A}}{3}\ge e^{-  N\left[2\int \log (\lambda_1-x)d\mathsf{MP}_{y}(x)-1\right]+O\left({\log N^{c\log \log N}}\right)},
\end{align}
where the last inequality uses \eqref{approx12} and \eqref{approx13}. 
For the $M=N$ case, as discussed right after \eqref{approx12}, applying Theorem \ref{t:gotzeks} and \eqref{e:ksapp} allows us to replace the  $O\left({\log N^{c\log \log N}}\right)$ term in \eqref{lb123} by $O(1).$ 
\end{proof}

We are now ready to prove Theorem \ref{t:ldpgen}.

\begin{proof}[Proof of Theorem \ref{t:ldpgen}]
Let $H_1(\lambda_1)$  denote the LHS  of \eqref{e:step1} i.e.,
\begin{align}\label{def30}
H_1(\lambda_1)&:=\int_{\underline{\lambda}^{(1)}: \lambda_2 \leq \lambda_1} V(\lambda_1; \underline{\lambda}^{(1)})^2 e^{-\sum_{i=2}^N\lambda_i}f_{M-1,N-1} d\underline{\lambda}^{(1)}.
\end{align}
It follows from \eqref{e:density} that 
\begin{equation}\label{total12}
\P(\lambda_1> (4+\delta))= \frac{Z_{M-1,N-1}}{Z_{M,N}}\int_{(4+\delta)}^{L} e^{-N(y^{-1}\lambda_1-(y^{-1}-1)\log \lambda_1)} H_1(\lambda_1)~d\lambda_1 +\P(\lambda_1> L).
\end{equation}
Using Theorem \ref{t:ldp}, it follows that we can ignore the second term in the above display as it is exponentially smaller than the first term, if $L$ is sufficiently large. Using \eqref{e:step1}, Lemma \ref{l:intapp} and Lemma \ref{l:partition} it follows that
\begin{align}\label{final1}
\P(\lambda_1\in (4+\delta,L)) \leq \int_{(4+\delta)}^{L} \exp \left(-N I_{y}(\lambda_1-4) +O(\log N^{c\log \log N}) \right)~d\lambda_1.
\end{align} 
The upper bound in the theorem follows just by observing that $I_{y}(\cdot)$ is increasing in $\lambda_1$ and the $O(\cdot)$ term is uniform.
For the lower bound, observe using \eqref{e:step2} and \eqref{lb1} we get 
\begin{align}
\label{final2}
\P(\lambda_1\in (4+\delta,L)) &\geq \int_{(4+\delta)}^{L} \exp \left(-N I_{y}(\lambda_1-4) +O(\log N^{c\log \log N}) \right)~d\lambda_1,\\
&\geq \int_{(4+\delta)}^{(4+\delta)+1/N} \exp \left(-N I_{y}(\lambda_1-4) +O(\log N^{c\log \log N}) \right)~d\lambda_1,\\
&\ge \exp \left(-N I_{y}(\delta) +O(\log N^{c\log \log N}) \right).
\end{align}
\end{proof}

Next we supply the extra ingredients needed for the proof of Theorem \ref{t:ldpm=n}.

\begin{proof}[Proof of Theorem \ref{t:ldpm=n}]
The proof follows the same lines of the proof of Theorem \ref{t:ldpgen} by using the $M=N$ case of Lemma \ref{l:intapp} in \eqref{final1} and \eqref{final2}. The goal is to obtain the $-\log N$ term in the statement of the theorem. To do this we discretize the integral appearing in \eqref{final1} and \eqref{final2} in steps of $1/N,$ i.e., for 
\begin{align}\label{discretize}
\P(\lambda_1\in (4+\delta,L))\le \sum_{j=0}^{N(L-(4+\delta))}\int_{(4+\delta)+\frac{j}{N}}^{(4+\delta)+\frac{j+1}{N}}\exp \left(-N I(\lambda_1-4) +O(1) \right)d\lambda_1
\end{align}
Using a uniform lower bound on the derivative of $I_{\delta}$ for $\delta>\delta_0$ it follows that 
\begin{align*} \int_{(4+\delta)+\frac{j}{N}}^{(4+\delta)+\frac{j+1}{N}} \exp \left(-N I(\lambda_1) +O(1) \right)~d\lambda_1 \leq \exp(-NI(\delta))\frac{1}{N} \exp (-{I'(\delta)j}).
\end{align*}
Thus summing the above, it follows that the RHS in \eqref{discretize} is bounded by $\exp(-NI(\delta)-\log N+O(1)).$
The lower bound follows from the same argument using the $M=N$ case for \eqref{final2}.
\end{proof}
Using similar arguments we now finish the proof of Lemma \ref{l:toptwo}
\begin{proof}[Proof of Lemma \ref{l:toptwo}] 
The argument is similar to the proof of Theorems \ref{t:ldpgen} except we now decompose the joint density of $n$ eigenvalues in to the top two and the last $n-2$.  Similar to \eqref{def30}, define
\begin{align*}
H_{1,2}(\lambda_1,\lambda_2)&:=\int_{\underline{\lambda}^{(1,2)}} V(\lambda_1,\lambda_2; \underline{\lambda}^{(1,2)})^2 e^{-2\sum_{i=2}^N\lambda_i}f_{M-2,N-2} d\underline{\lambda}^{(1,2)}, 
\end{align*}
where $\underline{\lambda}^{(1,2)}:= (\lambda_3, \ldots , \lambda_{N})$ and $$V(\lambda_1,\lambda_2; \underline{\lambda}^{(1,2)}):=\prod_{j\neq 1,2}(\lambda_1-\lambda_j)\prod_{j\neq 1,2}(\lambda_2-\lambda_j).$$
For the moment fix $\lambda_1, \lambda_2$ satisfying 
\begin{align*}
\lambda_1\ge \lambda_2,
\lambda_1\ge (4+\delta),
\lambda_2 \ge (4+\frac{\delta}{2}).
\end{align*}
Similar to \eqref{lipconst1} we will now need to choose two functions  $g_{\lambda_1}(\cdot)$ and $g_{\lambda_2}(\cdot).$  We choose $g_{\lambda_1}$ exactly as in \eqref{lipconst1}. 
$g_{\lambda_2}$ is now chosen by replacing $\delta$ by $\delta/2$ in the definition of the former, i.e.,   
\begin{equation}\label{choice2}
g_{\lambda_2}(x)=\left\{\begin{array}{cc}
2\log (\lambda_2-x)-x &\text{ for } x\le 4+\frac{\delta}{4}\\
g_{\lambda_2}(4+\frac{\delta}{4}),& \text{otherwise}.
\end{array}\right.
\end{equation}

 We now get the following bound corresponding to \eqref{e:step1},
\begin{equation}
\label{e:step101}
H_{1,2}(\lambda_1,\lambda_2) \leq \int_{\underline{\lambda}^{(1,2)}} e^{\sum_{i=3}^N \left(g_{\lambda_1}(\lambda_i)+g_{\lambda_2}(\lambda_i)\right)}f_{M-2,N-2} d\underline{\lambda}^{(1,2)}. 
\end{equation}

The next result similar to Lemma \ref{l:intapp} with the exact same proof now gives a bound on the RHS of \eqref{e:step101}.
\begin{align}\label{est10}
\log(H_1(\lambda_1,\lambda_2))&=2N [\int \log (\lambda_1-x)d\mathsf{MP}(x)+\int \log (\lambda_2-x)d\mathsf{MP}(x)]-2N+O(1).
\end{align}

As in \eqref{total12}, for any chosen $L>0,$
\begin{align}\label{total13}
&\P(\lambda_1> (4+\delta),\lambda_2\ge (4+\frac{\delta}{2}) )\\
\nonumber
&\le  \frac{Z_{M-2,N-2}}{Z_{M,N}}\int_{\lambda_1\ge \lambda_2, 
(4+\delta)\le \lambda_1\le L,\lambda_2\ge (4+\frac{\delta}{2})} |\lambda_1-\lambda_2|e^{-N(\lambda_1+\lambda_2)} H_{1,2}(\lambda_1,\lambda_2)~d\lambda_1d\lambda_2 +\P(\lambda_1> L),\\
& =  \frac{Z_{M-2,N-2}}{Z_{M,N}}O(L)\int_{\lambda_1\ge \lambda_2, 
(4+\delta)\le \lambda_1\le L,\lambda_2\ge (4+\frac{\delta}{2})}e^{-N(\lambda_1+\lambda_2)} H_{1,2}(\lambda_1,\lambda_2)~d\lambda_1d\lambda_2 +\P(\lambda_1> L).
\end{align}
Using Lemma \ref{l:partition}, plugging in \eqref{est10} into above and choosing $L$ to be a large enough constant dependent on $\delta$ we get that 
$$
\P(\lambda_1\ge 4+\delta, \lambda_2\ge 4+\delta/2)\le e^{-N\left(I(\delta))+I(\frac{\delta}{2})\right)+O(1)}.
$$  
 The proof of the lemma is now complete by comparing the above upper bound to $\P(\lambda_1\ge (4+\delta))$ from Theorem \ref{t:ldpm=n}.
\end{proof}
\section{Comparison of Rate Functions}
\label{s:polyub}
Throughout we have been assuming $M-N=o(N)$ but in this section we will pin down quantitative dependence of our estimates on $M-N.$ In particular we analyze how the rate function $I_y(\cdot)$ depends on $y$. Throughout the following discussion for the ease of notation let $n=\frac{N}{2}, m_1=n+c, n_1=n-c$ for some non-negative integer $c.$ 
For notational simplification, $\P_{m,n}$ will denote the probability measure induced by the Wishart matrix of dimensions $m\times n.$
We now proceed to compare $\P_{m_1,n_1}(m
_1 \lambda_1\ge (4+\delta)n)$ and $\P_{n,n}(n
 \lambda_1\ge (4+\delta)n)$ for their direct relations to transversal fluctuations of polymers.  Also recall $g(n)$ defined right before the statement of Lemma \ref{l:approx}.
Let $y=\frac{n_1}{m_1}$ and  $\hat \delta$ be such that 
\begin{equation}\label{constraint1}
(4+\delta)n=(4+\hat \delta )m_1.
\end{equation}
Note that this implies $$\hat \delta=\delta-\frac{(1-y)}{2}(4+\delta)=\delta-(4+\delta)\frac{c}{n}+O(\frac{c^2}{n^2}).$$
We now state the key proposition of this section.
\begin{prp}\label{curvature}$$\P_{m_1,n_1}\left(\lambda_1\ge (4+\hat \delta)\right)=\P_{n,n}\left(\lambda_1\ge (4+ \delta)\right)e^{-\beta_{\delta}(\frac{c^2}{n})+O(\frac{c^3}{n^2}+g(n))},$$ where 
$\beta_{\delta}=-6-\int \log(4+ \delta-x){\rm d}\mathsf{MP}
+(6+\delta) \int \frac{1}{4+\delta-x} d\mathsf{MP} +2\int_0^{4} \frac{\log (4+\delta-x)}{2\pi \sqrt{x(4-x)}}~dx.
$
\end{prp}

The proof of the above involves several steps of  analysis of the rate functions and for the ease of reader we break it into several steps.
By the proof of Theorem \ref{t:ldpgen}, it follows that  
$$
\P_{m_1,n_1}\left(\lambda_1\ge (4+\hat \delta)\right)=e^{A_0+A_1+A_2+A_3+A_4+O(g(n))},
$$
where 
\begin{align*}
A_0 &=\log \frac{Z_{m_1-1,n_1-1}}{Z_{m_1,n_1}},
A_1=\left(2n_1 \int \log(4+\hat \delta-x){\rm d}\mathsf{MP}_y \right),\,\,\\
A_2&=\left(-n_1 \int x{\rm d}\mathsf{MP}_y\right),
A_3 = (m_1-n_1) \log (4+\hat \delta),\,\, A_4 = -m_1 (4+\hat \delta).
\end{align*}

Similarly let $$
\P_{n,n}\left(\lambda_1\ge (4+ \delta)\right)=e^{B_0+B_1+B_2+B_3+B_4+O(\log n)},
$$
where $B_i$ is the analogous quantity to $A_i$ when $c=0$ i.e., $y=1.$ ($B_0:=\log(\frac{Z_{n-1,n-1}}{Z_{n,n}})$).
Thus the proof of Proposition \ref{curvature} will proceed by comparing   $A_0+A_1+A_2+A_3+A_4$ with $B_0+B_1+B_2+B_3+B_4$ by analyzing: 

\begin{align}
\label{a0}
A_{0}-B_1&=\log \frac{Z_{m_1-1,n_1-1}}{Z_{m_1,n_1}}-\log \frac{Z_{n-1,n-1}}{Z_{n,n}},\\
\label{a1}
A_1-B_1&= \left(2n_1 \int \log(4+\hat \delta-x){\rm d}\mathsf{MP}_y \right)-\left(2n \int \log(4+ \delta-x){\rm d}\mathsf{MP} \right),\\
\label{a2}
A_2-B_2&=\left(-n_1 \int x{\rm d}\mathsf{MP}_y\right)+\left(n \int x{\rm d}\mathsf{MP}\right),\\
\label{a3}
A_3-B_3&=(m_1-n_1) \log (4+\hat \delta),\\
\label{a4}
A_4-B_4 &= -m_1 (4+\hat \delta)+m (4+ \delta).
\end{align}

Although in principle the above bounds become better as $c$ increases, the computations cannot simply rely on perturbative arguments involving Taylor expansion for $c$. However for our purposes we only need to show that the bounds do not deteriorate. Thus the abstract coupling result stated in Theorem \ref{t:dom} suffices. This is proved in the following section using abstract facts about general  determinantal point process and could be of independent interest. The precise technical details for the proof of Proposition \ref{curvature} are provided in the Appendix (Section \ref{calculations}).

\section{Stochastic Inequalities for Point Processes}\label{secdom}
We will recall some basic facts, terminologies and notations about finite rank determinantal point processes on $\R_{+}$. However we will be brief in our treatment and we refer the interested reader to \cite{lyons} for a well rounded survey on the subject.
Let $\mu$ be a probability measure on $\R_{+}$. A collection $\Omega=(X_1,X_2,\ldots, X_{n})$ of $n$ random points on $\R_{+}$ is said to be a determinantal process with kernel $K$ (and background measure $\mu$) if there is a projection kernel $K$ of rank $n$ (i.e., $K(x,y)=\sum_{i=1}^{n} \phi_{i}(x)\phi_{i}(y)$ for an orthonormal set $\{\phi_1,\phi_2,\ldots, \phi_{n}\}$ in $L^2(\R_{+},\mu)$) such that the joint density of $(X_1,X_2,\ldots, X_{n})$ with respect to the product measure $\mu^{\otimes n}$ is proportional to 
$$ f(x_1,x_2, \ldots, x_{n})= \det[K(x_i,x_{j})_{1\leq i,j \leq n}].$$ 

Let $\phi_{\frac{M+N}{2}-1}, \phi_{\frac{M+N}{2}-2}, \ldots , \phi_{0}$ be the orthonormal functions in $L^2(\R_{+},\mathsf{MP})$ obtained by applying Gram-Schmidt procedure to the functions $x^{\frac{M+N}{2}-1}, x^{\frac{M+N}{2}-2}, \ldots, x, 1$ in that order. (Note that this implicitly assumes $M-N$ is even. For $M-N$ odd, one looks at polynomials with half integer degrees. $x^{1/2},x^{3/2} \ldots.$  Since our application will only involve the former case we will provide arguments only in the case $M-N$ being even.)
Consider the Kernel 
\begin{equation}\label{kernel1}
K_{M,N}(x,y)= \sum_{i=\frac{M-N}{2}}^{\frac{M+N}{2}-1} \phi_{i}(x)\phi_{i}(y),
\end{equation}
i.e., the projection kernel that projects onto the $N$ dimensional subspace of $L^2(\R_{+},\mathsf{MP})$ generated by the functions $x^{\frac{M-N}{2}}, x^{\frac{M-N}{2}+1}, \ldots , x^{\frac{M+N}{2}-1}.$ 
The following fact is well-known, and can easily be derived from the joint density of eigenvalues of Wishart matrix. 

\begin{theorem}
\label{t:det}
Let $M-N$ be even and let $X_{M\times N}$ ($M\geq N$) be a matrix of i.i.d.\ standard complex Gaussian entries. Then the eigenvalues $(\lambda_1,\lambda_2,\ldots, \lambda_{N})$ of $X^*X$ forms a determinantal point process on $\R_{+}$ with background measure $e^{-x}dx$, and projection kernel $K_{M,N}$.
\end{theorem}

\begin{proof} Recall from \eqref{e:density} that for any unnormalized vector $\underline{\lambda}$ the  eigenvalue density of $X^*X$ at $\underline \lambda$ is $$g(\underline{\lambda})\propto V(\underline{\lambda})^2\prod_{i=1}^{N}\lambda_i^{M-N}e^{-\sum_{i=1}^N\lambda_i} 
$$

Now for any $\underline \lambda=(\lambda_1,\lambda_2,\ldots, \lambda_n)$ an easy calculation using rules of computing determinants shows that  $${\rm{det}}\left[K(\lambda_i,\lambda_j)\right]_{i,j=1}^{n}=g(\underline \lambda) . 
$$ 
Thus an easy application of orthonormality of the $\phi_i$s and the Cauchy-Binet theorem (see \cite{lyons}) shows that $\underline \lambda$  is a determinantal point process.
 \end{proof}

Note also that \eqref{kernel1} implies that  $K_{M,N}\succeq K_{M+1,N-1}$ as operators where the inequality holds in the positive definite sense.  
We now quote the following result about abstract coupling of determinantal processes whose kernels are ordered in the positive definite sense. 

\begin{thm}\cite[Theorem 3.8]{lyons}If $K_1$ and $K_2$ are two locally trace class positive contractions on $L^2(\R_+,\mu)$ such that  $K_1\preceq K_2$ then $\cP_{K_1}\preceq \cP_{K_2}$ where $\cP_K$ denotes the determinantal point process associated with the kernel $K$ and $\preceq$ is used to denote the usual stochastic domination via the probabilities of increasing events.
\end{thm}
As a direct consequence of the above and the discussion preceding that we have the following stochastic ordering result. If $X_1$ and $X_2$ are two random matrices with standard complex gaussian entries of dimensions $m+1\times n-1$ and  $m\times n$ respectively with  $(\tilde{\lambda}_1\geq \tilde{\lambda}_2 \geq \cdots \geq \tilde{\lambda}_{n-1})$ and $(\tilde{\lambda}_1\geq \tilde{\lambda}_2 \geq \cdots \geq \tilde{\lambda}_{n})$ being the ordered eigenvalue sequence of $X_1^*X_1$ and $X_2^*X_2$ respectively following laws $\P_{m_1,n_1}$ and $\P_{m,n}$.
\begin{cor}$\P_{m_1,n_1}\preceq \P_{n,n}$ and hence in particular $\lambda_1$ stochastically dominates $\tilde \lambda_1.$
\end{cor}
Given the above preparations we are now ready to prove Theorems \ref{t:mainu} and \ref{t:mainl}.
\section{Proof of Theorem \ref{t:mainu}: Transversal Fluctuation Upper Bound}\label{ptu}
Throughout this section, for notational convenience we shall assume  that $n$ is even.  The reader will notice that the proofs go through verbatim for $n$ odd by considering the polymer to $n+1$ instead. We will also denote $(\log n)^{C\log \log n}$ by $h(n)$ where $2C$ is bigger than $c$ appearing in the definition of $g(n)$ defined right before Lemma \ref{l:approx}.
We start by outlining the basic steps. 
For $v\in \llbracket 1,n \rrbracket ^2$, let $\Gamma_{n}(v)$ denote the maximal weight path from $\mathbf{1}$ to $\mathbf{n}$ passing through $v$ and let $\mathcal{R}_{n}$ denote the set of all vertices $v=(v_1,v_2)\in \llbracket 0,n \rrbracket ^2$ such that $|v_1-v_2| \ge  n^{1/2}h(n)$. 
\begin{figure}[h]
\centering
\includegraphics[scale=.8]{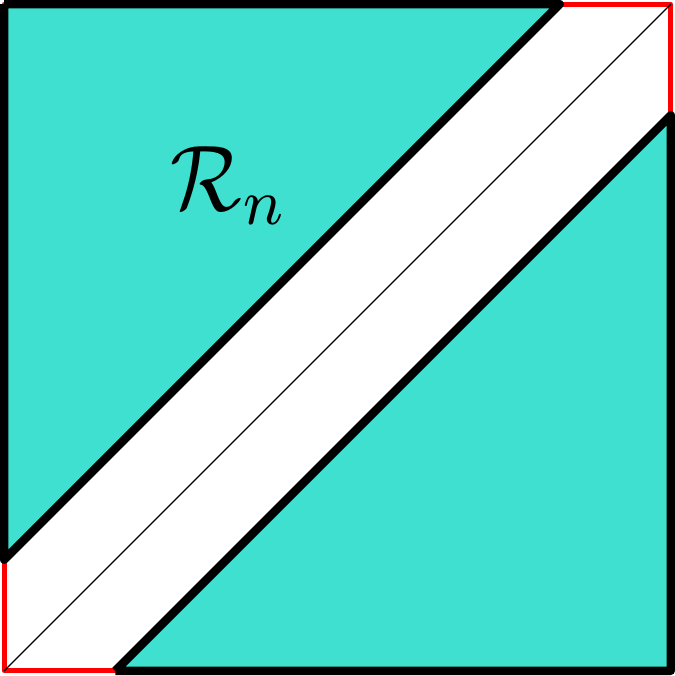}
\caption{Figure illustrating the region $\mathcal{R}_n$, described above which the polymer conditioned on the upper tail event, avoid.
}
\end{figure}
Clearly it suffices to show that 
$$ \sum_{v\in \mathcal{R}_{n}} \dfrac{\P(\ell (\Gamma_{n}(v))\geq (4+\delta)n)}{\P(T_{n}\geq (4+\delta)n)}=o(1).$$

We shall control this sum by separately summing over $v\in \mathcal{R}_{n}(t)$ where $t\in [2n]$, $\mathcal{R}_{n}(t): \mathcal{R}_{n}\cap \mathbb{L}_{t}$ where $\mathbb{L}_{t}$ denotes the line 

\begin{equation}\label{antidiag} \mathbb{L}_{t}:=\{(v_1,v_2)\in \Z^2: v_1+v_2=t\}.
\end{equation}

Recall from \eqref{transversal1} that $D_n(t)$ is the transversal fluctuation of the geodesic $\Gamma_n$ at time $t$.  
Since $\Gamma_n$ is a nearest neighbor path, we have $D_n(t)\le t$ for all $t,$ hence it suffices to prove the following proposition.

\begin{proposition}
\label{p:line}
For each $t\in \llbracket \frac{n^{1/2} h(n)}{2}, n-\frac{n^{1/2} h(n)}{2}\rrbracket$, we have 
$$\sum_{v\in \mathcal{R}_{n}(2t)} \dfrac{\P(\ell (\Gamma_{n}(v))\geq (4+\delta)n)}{\P(T_{n}\geq (4+\delta)n)}=o(n^{-1}).$$
\end{proposition}

For the remainder of this section, let us fix $t\in \llbracket \frac{n^{1/2} h(n)}{2}, n-\frac{n^{1/2} h(n)}{2}\rrbracket$. We first make the following basic observation which will be useful. 

\begin{observation}
\label{o:dom}
With the above notation; $\ell (\Gamma_{n}(v))$ is stochastically dominated by $T_{\mathbf{1},v}+T'_{v,\mathbf{n}}$ where $T'_{v,\mathbf{n}}$ is an independent copy of $T_{v,\mathbf{n}}$.
\end{observation}
\begin{proof} Note that $\ell (\Gamma_{n}(v))=T_{\mathbf{1},v}+T_{v,\bn}-X_v$ since $X_v$ (the exponential variable corresponding to $v$ see Section \ref{setting}) contributes to both $T_{\mathbf{1},v}$ and $T_{v,\bn}$. Thus $\ell (\Gamma_{n}(v))\preceq T_{\mathbf{1},v}+T_{v,\bn}-X_v+X'_v$ where $X_v$ is an independent copy of $X'_v.$ Now the proof is complete by noticing that $T_{v,\bn}-X_v+X'_v$ has the same law as $T_{v,\bn}$ and is independent of $T_{\mathbf{1},v}.$
\end{proof}

Now the monotonicity result in Theorem \ref{t:dom} and the obvious symmetry about the diagonal, immediately leads to the following observation, which will reduce the task of proving Proposition \ref{p:line} to proving it for one choice of $v$.

\begin{observation}
\label{o:dom2} 
Fix $t\in \llbracket n^{1/2}\frac{h(n)}{2}, n-n^{1/2}\frac{h(n)}{2}\rrbracket$. For each $v\in \mathcal{R}_n(2t)$, $\ell(\Gamma_{n}(v))$ is stochastically dominated by  $T_{\mathbf{1},v_0}+T'_{v_{0},\mathbf{n}}$  where $v_0=v_0(t):=(t+\frac{n^{1/2}h(n)}{2}, t-\frac{n^{1/2}h(n)}{2})$.
\end{observation}

Observation \ref{o:dom2} reduces Proposition \ref{p:line} to the following. 

\begin{proposition}
\label{p:line2}
Fix $t\in  \llbracket n^{1/2}\frac{h(n)}{2}, n-n^{1/2}\frac{h(n)}{2}\rrbracket$. For $v_0$ as above we have, 
\begin{equation}
\label{e:line2}
\dfrac{\P(T_{\mathbf{1},v_0}+T'_{v_{0},\mathbf{n}} \geq (4+\delta)n)}{\P(T_{n}\geq (4+\delta)n)}=o(n^{-2}).
\end{equation}
\end{proposition}
\begin{figure}[h]
\centering
\includegraphics[scale=.5]{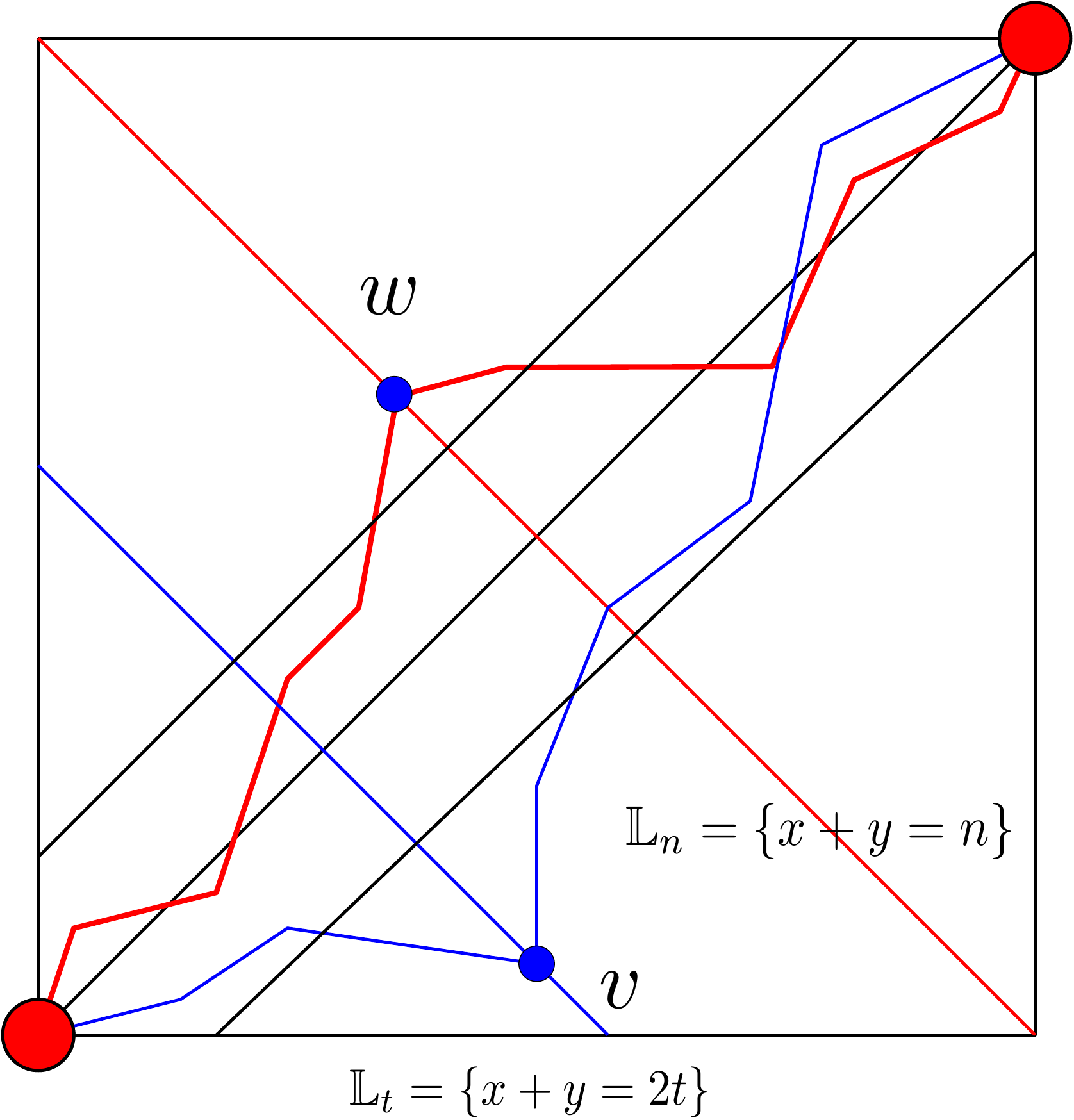}
\caption{ $v$ and $w$ are two point in $\mathcal{R}_n(2t)$ and $\mathcal{R}_n(n)$  respectively. The blue and the red paths denote the best paths passing through $v$ and $w$. Proposition \ref{p:line} shows that, conditioned on the upper tail event, no such path is likely to be the global polymer.
}
\end{figure}

We comment at this point that the reduction above is not a necessity, we shall prove upper bounds on $\P(T_{\mathbf{1},v}+T'_{v,\mathbf{n}} \geq (4+\delta)n)$ that gets progressively worse as $v$ moves away from the diagonal on $\mathbb{L}_{t}$; however for a technical convenience it will be easier to restrict our argument for $v=(v_1,v_2)$  where $|v_1-v_2|=o(t\wedge (n-t))$. 
Observe now that, using the sharp non-asymptotic lower bound for $\P(T_{n}\geq (4+\delta)n)$ obtained in Theorem \ref{t:ldpm=n}, we shall need the following upper bound of the numerator of \eqref{e:line2}. 

\begin{proposition}
\label{p:line3}
In the setting of Proposition \ref{p:line2}, we have 
$$\log \P(T_{\mathbf{1},v_0}+T'_{v_{0},\mathbf{n}} \geq (4+\delta)n) \leq -n I (\delta)-\Theta(h(n)).$$
\end{proposition}
We first finish the proof of Proposition \ref{p:line2} using the above. 

\begin{proof}[Proof of Proposition \ref{p:line2}] Follows by plugging in the conclusion of Theorem \ref{t:ldpm=n} and the above upper bound and noticing that $e^{-h(n)}$ decays to zero super-polynomially in $n.$
\end{proof}

{We now prove Proposition \ref{p:line3}}.
\begin{proof}[Proof of Proposition \ref{p:line3}]
Towards the proof of Proposition \ref{p:line3}, let us set $\e_{n}:=\frac{1}{n^2}$ and for $j\in \Z$ let $\cg_{j}$ (resp.\ $\cg'_{j}$) denote the event that $T_{\mathbf{1},v_0}\ge (4+\delta)t+j \e_n$ (resp.\ $T'_{v_{0},\mathbf{n}}\ge (4+\delta)(n-t)+j'\e_n$).
 Clearly, 
\begin{equation}\label{mesh}
 \P(T_{\mathbf{1},v_0}+T'_{v_{0},\mathbf{n}} \geq (4+\delta)n) \leq \sum_{j+j'\geq -1} \P(\cg_{j})\P(\cg'_{j'}).
 \end{equation}

Moreover let $j_{\max}$ be such that $j_{\max}\e_n=(4+\delta)n $ and $j_{\min}=-j_{\max}.$
We now see that we can restrict the above sum to $j_{\min}\le j,j'\le j_{\max}.$ Thus there are $O(n^3)$ terms in the sum.  Clearly, there is no loss of generality is assuming $t\leq \frac{n}{2}$. We first prove Proposition \ref{p:line3} in the special case of $t=\frac{n}{2}.$ We will then need some extra estimates to extend the same proof for general $t$.

Observe that by the strict convexity of the rate function $I(\cdot)$ we know that there exists $\varepsilon=\e(\delta)>0$ such that $I(2\delta-\varepsilon)>2I(\delta)$. Let $j_*$ be such that $(j_*-1)\e_{n}= (\delta-\e)n/2$. It follows that if $\max(j,j')\geq j_*$ then 
\begin{equation}\label{convex1}
\log \P(\cg_{j})\P(\cg'_{j'}) \le  -n(I(\delta)+c),
\end{equation}
for some $c>0$ and hence those terms contribute at most $e^{-n(I(\delta)+c)}O(n^3)$ to the sum in \eqref{mesh}. Thus these can be ignored for the purpose of the proof of Proposition \ref{p:line3}. So from now on we shall restrict ourselves to the case where $\max(j,j')\leq j_{*}$.
Let us also introduce the following notations: 
$$ N_1:=\frac{n}{2}-n^{1/2}h(n); \qquad M_1:=\frac{n}{2}+n^{1/2}h(n); \qquad y:= \frac{N_1}{M_1}.$$
Fix $(j,j')$ such that $j+j'\geq -1$ and $\max(j,j')\leq j_{*}$. Note that, by our choice of $j_{\max}$ and the constraint on $(j,j')$ it follows that $\delta + (\min(j,j'))\frac{2\e_{n}}{n}> \e$ and $\delta + (\max(j,j'))\frac{2\e_{n}}{n}\le 2\delta$ and we can apply Theorem \ref{t:ldpgen} to conclude that   
$$\log  \P(\cg_{j}) = -N_1 I_{y}(\delta + j\frac{2\e_{n}}{n})+ O (g(n));$$
$$\log  \P(\cg'_{j'}) = -N_1 I_{y}(\delta + j'\frac{2\e_{n}}{n})+ O (g(n)).$$

Again, to reduce notational overhead let us set $\delta_{j}:= \delta + j\frac{2\e_{n}}{n}$.  Now by Proposition \ref{curvature} $$ N_1 I_{y}(\delta_{j}) \geq \frac{n}{2} I(\delta_{j}) + \Theta(h^2(n)).$$
 This implies,
$$\log \P(\cg_{j})\P(\cg'_{j'}) \leq -n \biggl ( \frac{I(\delta_j)+ I(\delta_{j'})}{2} \biggr)- \Theta(h^2(n)) + O (g(n)).$$

By convexity and monotonicity of $I(\cdot)$, and our choice of $(j,j')$

$$ \frac{I(\delta_j)+ I(\delta_{j'})}{2} \geq I(\delta_{-1}) \geq I(\delta) -c \frac{\e_{n}}{n}$$ 
for some $c$ bounded away from $\infty$ (as $I(\cdot)$ is continuously differentiable with derivative bounded by $2$). Since there are only $O(n^3)$ terms in \eqref{mesh} this completes the proof of Proposition \ref{p:line3} for the case $t=n/2$.

Note that in the above, $t=n/2$ was used implicitly to restrict the values of $j, j'$ to be less than $j_*.$
However when $t$ is much smaller, $j$ could potentially be much bigger.  This calls for some extra estimates that we present next. First of all recall that as  pointed out after \eqref{e:i}, the rate function $I_{\delta}$ is not quite strongly convex and the hessian decays to zero as $\delta$ approaches infinity. We record the following useful lemma. 

\begin{lem}\label{strongconvexity} Fix any $\delta>0$.  Then there exists $C=C_{\delta}>0$ such that uniformly for any $\alpha\in [0,1/2]$ and $\delta_1$ and $\delta_2$ such that $\alpha \delta_1+(1-\alpha)\delta_2=\delta$

$$ \alpha_1I(\delta_1)+(1-\alpha)I(\delta_2)-I({\delta})\ge C_{\delta}\left[(1-\alpha)(\delta_2-\delta)^2+\alpha [\min((\delta_1-\delta)^2, |\delta_1-\delta |)\right]$$
\end{lem}
\begin{proof} 
For the moment let us assume $\delta_1<\delta.$  which implies $\delta_2\le 2\delta.$ Then by Taylor expansion: $$I(\delta_1)\ge I(\delta)+(\delta_1-\delta)I'(\delta)+I''(\delta)(\delta_1-\delta)^2$$
Above we use that in fact $I''(\delta)$ is decreasing in $\delta.$ Now for $\delta_2,$ similarly we have $$I(\delta_2)\ge I(\delta)+(\delta_2-\delta)I'(\delta)+I''(\delta_2)(\delta_1-\delta)^2$$

Since $\delta_2<2\delta,$  we can take $C_{\delta}=I''(2\delta)$ and we are done.
However, when $\delta_1>\delta$, the situation is slightly more complicated since $\alpha$ can be really close to zero pushing $\delta_1$ towards infinity where the hessian becomes almost zero.  
If $\delta_1<2\delta$ then the above argument works. When $\delta_1>2\delta$ we use the following bound instead (using $I'(\delta)$ is increasing in $\delta,$)
\begin{align*}
I(\delta_1)&\ge I(\delta)+I'(\delta)(\frac{3\delta}{2}-\delta)+ I'(\frac{3\delta}{2})(\delta_1-\frac{3\delta}{2}), \text{ which by re-arranging}\\
&= I(\delta)+I'(\delta)(\delta_1-\delta)+ [I'(\frac{3\delta}{2})-I'(\delta)](\delta_1-\frac{3\delta}{2}).
\end{align*}
The above and the fact that by hypothesis $\delta_1-\frac{3\delta}{2}=\Theta (\delta_1-\delta),$
completes the proof in the case $\delta_1\ge 2\delta.$
\end{proof}

Let $v_t=(t/2,t/2)$ and let $L_1$ and $L_2$ be the polymer weights from $\bf 1$ and $\bf n$ respectively  to $v_t$. 
The next lemma uses the above lemma to show that for any $t_1=t$, the excess polymer weight is more or less 
distributed proportionally  between the polymer to the line $\{x
+y=2t_1\}$ and beyond.  Let $n-t_1=t_2$ and $L_1$ and $L_2$ denote $T_{\mathbf{1},v_t}$ and $T_{v_t,\mathbf{n}}$ respectively.
\begin{lem}\label{est98} Fix $\delta>0.$ Uniformly for any $\delta_1,\delta_2>0$ which satisfy $t_1\delta_1+t_2 \delta_2\ge n\delta,$ 
$$\P(L_1\ge (4+\delta_1)t_1)\P(L_2\ge (4+\delta_2)t_2)\le e^{-\left[I(\delta)n+ n\Theta\bigl( \frac{t_1}{t_1+t_2}(\min((\delta_1-\delta)^2,(\delta_1-\delta)))+ \frac{t_2}{t_1+t_2} (\delta_2-\delta)^2\bigr)\right]+O(4+\delta_1+\delta_2)}.$$
\end{lem}

\begin{proof}As a straightforward consequence of Proposition \ref{t:diaguniform} we get 
\begin{equation}\label{prod12}
\P(L_1\ge (4+\delta_1)t_1)\P(L_2\ge (4+\delta_2)t_2)\le e^{-t_1I(\delta_1)-t_2I(\delta_2)+O(4+\delta_1+\delta_2)}.
\end{equation}  
Now notice that $t_1\delta_1+t_2 \delta_2\ge (t_1+t_2)\delta.$ So by Lemma \ref{strongconvexity}, 
$$
\frac{t_1}{t_1+t_2}I(\delta_1)+\frac{t_2}{t_1+t_2}I(\delta_2)- I(\delta)\ge \Theta\left( \frac{t_1}{t_1+t_2}\min((\delta_1-\delta)^2,(\delta_1-\delta))+ \frac{t_2}{t_1+t_2} (\delta_2-\delta)^2\right).
$$ 
\end{proof}
We will now consider the above estimates in the context of the sum in \eqref{mesh} by taking $\delta_1$ and $\delta_2$ to be such that 
\begin{align*}
t_1\delta_1&=t_1 \delta+j \e_n,\\
t_2\delta_1&=t_2 \delta+j' \e_n.
\end{align*}
Since $t_2\ge \frac{n}{2},$ it follows that the above choices of $\delta_1$ and $\delta_2$ can contribute to the sum in \eqref{mesh} only when 
(considering the sum in \eqref{mesh} has only $O(n^3)$ many terms),
 \begin{equation}\label{regime12}
 (\delta_2-\delta)^2\le O(\frac{\log n}{n}) \text{ and } t_1(\delta_1-\delta)^2\le O(\log n).
\end{equation} 
 
  Now recall from Proposition \ref{p:line} that $t_1\ge \sqrt n$ and hence $(\delta_1-\delta)^2\le \frac{O(\log (n))}{n}.$ 
However even though so far we have focussed only on $v_{t}=(t/2,t/2)$  to finish the proof of Proposition \ref{p:line3}, we have to bound the expression 
$$\P(T_{\mathbf{1},v_0}\ge (4+\delta_1)t_1)\P(T_{v_{0},\mathbf{n}} \ge (4+\delta_2)t_2)$$  where $v_0=v_0(t):=(t+\frac{n^{1/2}h(n)}{2}, t-\frac{n^{1/2}h(n)}{2}).$
Just by the monotonicity result in Theorem \ref{t:dom}, the same bounds as in Lemma \ref{est98} and hence the conclusions in \eqref{regime12} continue to hold. 
Now  since both $\delta_1$ and $\delta_2$ are bounded away from zero and infinity, by Theorem \ref{t:ldpgen}, and Proposition \ref{curvature} we have the following conclusion from \eqref{prod12}. 

$$\P(T_{\mathbf{1},v_0}\ge (4+\delta_1)t_1)\P(T_{v_{0},\mathbf{n}} \ge (4+\delta_2)t_2)\le e^{-I(\delta)n-\Theta(t_1(
\delta_1-\delta)^2+t_2(\delta_2-\delta)^2)-\Theta(h(n))}.$$

The above estimate along with \eqref{mesh},  completes the proof of Proposition \ref{p:line3}.
\end{proof}

\section{Proof of Theorem \ref{t:mainl}: Transversal Fluctuation Lower Bound}\label{ptl}
The proof will rely on the sharp on-diagonal large deviation in Theorem \ref{t:ldpm=n} and the monotonicity result Theorem \ref{t:dom} which will allow us to bound the off-diagonal terms by the on diagonal term. 
As outlined before, the main work in proving Theorem \ref{t:mainl} goes into proving the following proposition.

\begin{proposition}
\label{p:midpoint}
Fix $\delta>0$. There exists a constant $C=C(\delta)>0$ such that we have for all $n$ sufficiently large 
$$\P(T_{\mathbf{1},v_*}+T'_{v_{*},\mathbf{n}})\geq (4+\delta)n\mid \cU_{\delta}(n)) \leq \frac{C}{\sqrt{n}}$$
where $v_*=(\frac{n}{2}, \frac{n}{2})$.  
\end{proposition} 
\begin{figure}[h]
\centering
\includegraphics[scale=.8]{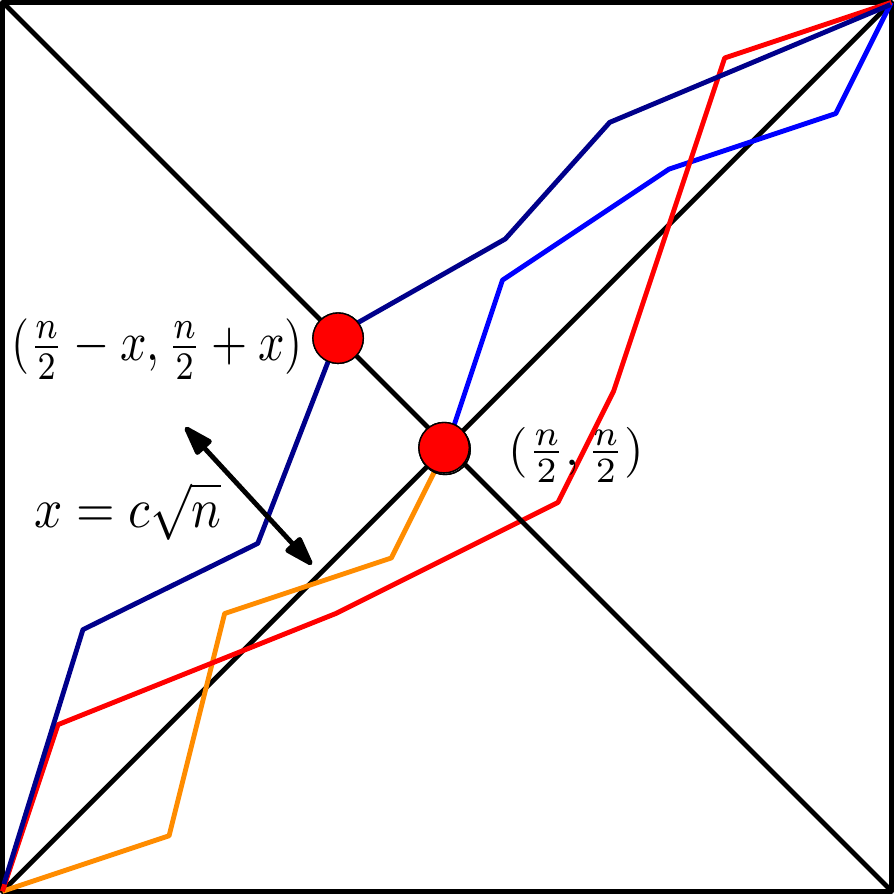}
\caption{Figure illustrating the proof of Theorem \ref{t:mainl}. The red path denotes the actual polymer conditioned on the upper tail event. We show that the best path passing through $(n/2,n/2)$ only has a chance of $O(\frac{1}{\sqrt{n}})$ of being the polymer. The same conclusion holds for the path passing through $(\frac{n}{2}-x,\frac{n}{2}+x)$ for any $x$ by the monotonicity result Theorem \ref{t:dom}.}
\end{figure}
Using Proposition \ref{p:midpoint}, it is now easy to prove Theorem \ref{t:mainl}. 

\begin{proof}[Proof of Theorem \ref{t:mainl}]
We shall show that $\P(D_{n}(n)\leq cn^{1/2}\mid \cU_{\delta}(n))\to 0$ as $c\to 0$ which clearly suffices. 
For $v\in \mathbb{L}(n)$, let $A_{v}$ denote the event that $\Gamma_{n}$ passes through $v$. Let $\mathbb{L}^*(c,n)$ denote the set of all $v=(n/2+x.n/2-x)\in \mathbb{L}(n)$ such that $|x|\leq cn^{1/2}$. Hence it suffices to show that 
\begin{equation}
\label{e:lbsum}
\sum_{v\in \mathbb{L}^*(c,n)} \P(A_{v}\mid \cU_{\delta}(n))\to 0
\end{equation}
as $c\to 0$ for all $n$ sufficiently large. It is clear that for each $v\in L(n)$ we have
$$\P(A_{v}\mid \cU_{\delta}(n))\leq \P(\ell(\Gamma_{n}(v))\geq (4+\delta)n\mid \cU_{\delta}(n)).$$

Now it follows that 
\begin{align*}\P(\ell(\Gamma_{n}(v))\geq (4+\delta)n\mid \cU_{\delta}(n))&\le \frac{\P(T_{\mathbf{1},v}+T'_{v,\mathbf{n}}\geq (4+\delta)n)}{\P(\cU_{\delta}(n))},\\
& \le \frac{\P(T_{\mathbf{1},v_*}+T'_{v_*,\mathbf{n}}\geq (4+\delta)n)}{\P(\cU_{\delta}(n))}.
\end{align*}
where the first inequality follows from Observation \ref{o:dom} and the second inequality is a straightforward consequence of Theorem \ref{t:dom}.
The proof is now complete as \eqref{e:lbsum} follows from the above and Proposition \ref{p:midpoint}.
\end{proof}

We now provide the details of the proof of Proposition \ref{p:midpoint}.
\begin{proof}[Proof of Proposition \ref{p:midpoint}]
We first write down a variant of the sum in \eqref{mesh}.
\begin{equation}\label{mesh1}
 \P(T_{\mathbf{1},v_*}+T'_{v_{*},\mathbf{n}} \geq (4+\delta)n) \leq 2\sum_{i\ge 0:\delta_{-i}\ge \e }  \P(T_{\mathbf{1},v_*} \geq (4+\delta_{-i})n) \P(T'_{v_{*},\mathbf{n}} \geq (4+\delta_i)n).
 \end{equation}
 where $\delta_{i}=\delta+\frac{i}{n}$ and  by the discussion preceeding \eqref{convex1}, the above sum can be restricted to the case neither $\delta_{-i}$ is at least $\e$ for some $\e=\e(\delta)>0$. 
Using Theorem \ref{t:ldpm=n} we obtain the following bound on the  RHS of \eqref{mesh1}.
\begin{align*}
\sum_{i}e^{-[\frac{I(\delta_i)+I(\delta_{i-1})}{2}]n-2\log n+O(1)}&=O(1)e^{-\log n}\sum_{i}e^{-I(\delta)n-
I''(\delta)\frac{i^2}{n}-\log n}\\
&=O(1)e^{-\log n}\frac{e^{-I(\delta)n}}{\sqrt n}.
\end{align*} 
 Thus we are done by observing that the above along with Theorem \ref{t:ldpm=n} implies that, 
 \begin{align*}
\frac{\P(T_{\mathbf{1},v_*}+T'_{v_{*},\mathbf{n}} \geq (4+\delta)n)}{\P(T_{\bf n}\ge (4+\delta)n)} \le \frac{O(1)}{\sqrt{n}}.
\end{align*} 
 
\end{proof}
\section{Concluding remarks and future directions}\label{fd}
We include a discussion of possible extensions of Theorem \ref{t:main} and outline some directions of future research. The transversal fluctuation exponent $1/2$ we have established for the polymer in the upper tail large deviation regime should be universal for a large class of two dimensional models of last passage percolation. However, unlike the lower tail case in \cite{BGS17A}, since our current method relies on the connections to random matrices, extending the result to general classes of LPP models remains an open problem. Nonetheless, for certain exactly solvable models of last two dimensional last passage percolation, one hopes that our method can be pushed through to obtain the same exponent. Observe that the main ingredient for the upper bound was to obtain the sharp asymptotics of the large deviation probability, which in turn used the representation of the last passage time as the position of the top particle in a certain determinantal process on $\R$ (eigenvalues of LUE) and the explicit joint density of the particles of the same. The key property used in such context was the rigidity property of eigenvalues which is true for a vast class of determinantal processes (see for e.g. \cite{gp}). 
Thus we expect our proof approach to work for other integrable models of last passage percolation with such connections including the Brownian Last Passage Percolation and Poissonian Last passage Percolation or the problem of the longest increasing subsequence of random permutations. We elaborate further on this point below. 

The most natural setting to extend our result is that of Brownian last passage percolation or the so called O'Connell-Yor polymer (see \cite{OY02, H16} for a precise definition). It is well-known \cite{CH14} that in Brownian LPP the passage time $L_{n,t}$ from {$(0,0)$ to $(t,n)$} is equal in distribution to $\sqrt{t}\lambda_1$ where $\lambda_1$ is the largest eigenvalue of an $n\times n$ GUE matrix whose eigenvalues $\lambda_1>\lambda_{2}>\cdots > \lambda_{n}$ has joint density proportional to $V(\underline{\lambda})^2 e^{-\frac{1}{2}\sum \lambda_{i}^2}$. It is also well-known \cite{ADG} that $\frac{L_{n,n}}{2n}\to 1$ a.s. and $-\frac{\log \P(L_{n,n}\geq (2+\delta)n)}{n}$ converges to a rate function $\tilde I(\delta),$  which is explicit. Following the same line of arguments as in the proof of Theorem \ref{t:ldpm=n} and using rigidity results for GUE eigenvalues \cite{GT16} one should be able to show that $\P(L_{n,n}> (2+\delta)n)= -nI(\delta)-\log n+ O(1)$. Satya Majumdar has informed us that results of a similar flavor have been derived in the  physics literature using different methods in \cite{borot,nadal}.  Observing that $L_{n,t}$ has the same distribution as $\sqrt{t}L_{n,1}$ for every $n\in \N$ and $t>0$, and following the same line of arguments as in this paper one should be able to prove results analogous to Theorems \ref{t:mainu} and \ref{t:mainl} in this setting as well. However we do not pursue working the details out precisely in this paper. 

For the model of Poissonian last passage percolation on $\R^2$ where one studies the length of the longest increasing path in a Poisson point process, using the RSK correspondence, one can show that the maximum number of points in an increasing path has the same distribution as the size of the top row of a Young Tableaux drawn from the Poissonized Plancherel measure \cite{romik}. The sizes of the rows of such a Young Tableaux forms a discrete determinantal process and one might try to carry our strategy of proof in this case also, invoking rigidity results in such context.

One natural question to ask is if there exists a non-trivial weak scaling limit of the polymer conditional on the upper tail large deviation event. More precisely, consider the following. Recall that for $t\in {0,1,\ldots, 2n}$, $(\frac{t}{2}+D_{n}(t), \frac{t}{2}-D_{n}(t))$ denotes the unique point on the polymer on the anti-diagonal  line $x+y=t$. Consider the process $D^*_{n}(s):= \frac{\sqrt{D_{n}(s2n)}}{\sqrt{n}}$ extended to $s\in [0,1]$ by linear interpolation. Does  $D^*_{n}$ conditional on $\{T_{n}\geq (4+\delta)n\}$ converge to a Brownian bridge?  We shall take up this question in a future project where we plan to use various representations for the evolution of the polymer weight profile to compute the correlation structure of the polymer weight across different times conditional on the large deviation event.
A preliminary step is to show that the one point distribution of say  $D^*_{n}(\frac{1}{2})$ is given by a Gaussian Random variable.
We believe that it might already be possible to show using our techniques that the limit, if exists is H\"{o}lder $1/2-$. 
We also wish to point out that even for the typical behavior of the polymer, existence of such a scaling limit was open for a long time, and was established only very recently in a breakthrough work \cite{DOV18} starting with the model of Brownian LPP. They also establish that the scaling limit was H\"{o}lder $2/3-$, as is expected given the typical transversal fluctuation scaling (see also \cite{HS18}). The scaling limit there is described as a functional of the so-called Airy Sheet. However the distributional properties are not very explicit and hence a better understanding of such scaling limits remain an important research area. 

\section{Appendix}
\label{calculations}
We finish with the details of the Proof of Proposition \ref{curvature}. We start with the following key estimate. 
\begin{theorem}
\label{t:intest}
Let $z=1-\sqrt{y}$, and recall \eqref{constraint1}. Let 
$$I:= \int (4+\hat{\delta}-x)~d\mathsf{MP}_{y}-\int (4+{\delta}-x)~d\mathsf{MP}.$$
Then $I=Az+Bz^2+o(z^2)$ as $z\to 0$ where 
$$A= -1+ \int \log (4+\delta-x)~d\mathsf{MP}-\log (4+\delta);$$
$$B=-{\frac12}-\frac{3}{2}\log (4+\delta)+ (\frac{1}{2}(2+\delta)+2) \int \frac{1}{4+\delta-x} d\mathsf{MP} +\int \log (4+\delta-x)~d\mathsf{MP}+ \int_0^{4} \frac{\log (4+\delta-x)}{2\pi \sqrt{x(4-x)}}~dx.$$  
\end{theorem}

Using the standard change of variable that takes $\mathsf{MP}$ to $\mathsf{MP}_y$ ($x\to (1-z)x+z^2$) and substituting $\sqrt{y}=1-z$ we get 

$$I= \bigint_0^{4} \left[ \frac{\log \biggl( (4+\delta-x)(1-z+\frac{z^2}{2})+\frac{z^2(x-2)}{2} \biggr) }{(1-z)x+z^2} -\frac{\log (4+\delta-x)}{x}\right] \frac{\sqrt{x(4-x)}}{2\pi}~dx.$$

Write $I=I_1+I_2$ where 

$$I_1= \bigint_0^{4} \left[ \frac{\log \biggl( (4+\delta-x)(1-z+\frac{z^2}{2})+\frac{z^2(x-2)}{2} \biggr) }{(1-z)x+z^2} -\frac{\log (4+\delta-x)}{(1-z)x+z^2}\right] \frac{\sqrt{x(4-x)}}{2\pi}~dx;~\text{and}$$

$$I_2= \int_0^4 \left[ \frac{\log (4+\delta-x)}{(1-z)x+z^2}- \frac{\log (4+\delta-x)}{x}\right]\frac{\sqrt{x(4-x)}}{2\pi}~dx.$$

Theorem \ref{t:intest} will follow from the next two propositions. 

\begin{proposition}
\label{p:inti1}
As $z\to 0$, we have $I_1=A_1z+B_1z^2+o(z^2)$ where $A_1=-1$ and $B_1=-\frac{1}{2}+ \frac{1}{2}(2+\delta)\int \frac{1}{4+\delta-x} d\mathsf{MP}$.
\end{proposition}

\begin{proposition}
\label{p:inti2}
As $z\to 0$, we have $I_2=A_2z+B_2z^2+o(z^2)$ where
$$A_2= \int \log (4+\delta-x)~d\mathsf{MP}-\log (4+\delta);$$
$$B_2=-\frac{3}{2}\log (4+\delta)+ 2\int \frac{1}{4+\delta-x} d\mathsf{MP} +\int \log (4+\delta-x)~d\mathsf{MP}+ \int_0^{4} \frac{\log (4+\delta-x)}{2\pi \sqrt{x(4-x)}}~dx.$$
\end{proposition}

We first prove Proposition \ref{p:inti1}.

\begin{proof}[Proof of Proposition \ref{p:inti1}]
Observe that $$I_1= \bigint_0^{4} \left[ \frac{\log \biggl( (1-z+\frac{z^2}{2})+\frac{z^2(x-2)}{2(4+\delta-x)} \biggr) }{(1-z)x+z^2}\right] \frac{\sqrt{x(4-x)}}{2\pi}~dx.$$
Call the term inside the bracket $S$. Observe that the numerator of $S$ if of the form $\log (1+w)$ where $|w|<1$ for $z$ small enough. By Taylor expanding the numerator of $S$ at $z=0$ we get 

$$S= \frac{-z+\frac{z^2(x-2)}{2(4+\delta-x) }+O(z^3)}{(1-z)x+z^2}$$
where the constant in the $O(\cdot)$ term depends on $x$ but is uniformly bounded above for $x\in [0,4]$. This together with the observation that the denominator of $S$ is $\Theta(x)$ for $z$ small and that $\sqrt{(4-x)/x}$ is integrable at 0 implies 

$$I_1= z\bigint_0^{4} \left[\frac{-1+\frac{z(x-2)}{2(4+\delta-x) }}{(1-z)x+z^2}\right] \frac{\sqrt{x(4-x)}}{2\pi}~dx +O(z^3).$$
We now write $I_1=I_{11}+I_{12}+O(z^3)$ where 

$$I_{11}= -z \int_{0}^{4} \frac{1}{(1-z)x+z^2}\frac{\sqrt{x(4-x)}}{2\pi}~dx;~\text{and}$$
$$I_{12}= z^2 \int_{0}^{4}\frac{\frac{(x-2)}{2(4+\delta-x)}}{(1-z)x+z^2}\frac{\sqrt{x(4-x)}}{2\pi}~dx.$$
To control $I_{11}$ and $I_{12}$ we need the following fact:
\begin{equation}
\label{e:wolf1}
\int_{0}^{4} \frac{1}{(ax+b)}\frac{\sqrt{x(4-x)}}{2\pi}~dx= \frac{-\sqrt{4ab+b^2}+2a+b}{2a^2}
\end{equation}
if $a,b>0$. Using this with $a=(1-z)$ and $b=z^2$ gives 
$$I_{11}=-z \frac{z(z-2)+2(1-z)+z^2}{2(1-z)^2}= -z.$$
Now moving on to $I_{12}$, observe that simple algebra gives 

$$ I_{12}= -\frac{z^2}{2} \int_{0}^{4} \frac{\sqrt{x(4-x)}}{2\pi((1-z)x+z^2)}~dx + \frac{z^2(\delta+2)}{2} \int_{0}^{4} \frac{1}{(4+\delta-x)}\frac{\sqrt{x(4-x)}}{2\pi((1-z)x+z^2)}~dx.$$
Using \eqref{e:wolf1} for the first term and observing that we have (by DCT) 
$$\int_{0}^{4} \frac{1}{(4+\delta-x)}\frac{\sqrt{x(4-x)}}{2\pi((1-z)x+z^2)}~dx \to \int \frac{1}{(4+\delta-x)}~d\mathsf{MP}$$
yields
$$ I_{12}=-\frac{z^2}{2}+ z^2\frac{\delta+2}{2}\int \frac{1}{(4+\delta-x)}~d\mathsf{MP}+ o(z^2)$$
as $z=0$. Combining with the above gives
$$I_1=-z+\left(-\frac{1}{2}+\frac{\delta+2}{2}\int \frac{1}{(4+\delta-x)}~d\mathsf{MP}\right)z^2+o(z^2)$$
as $z\to 0$ completing the proof.
\end{proof}

We now move towards the proof of Proposition \ref{p:inti2} which is more involved and we need a number of preparatory lemmas. 

\begin{lemma}
\label{l:i21}
Let 
$$I_{21}= \int_{0}^{4} \frac{\log (4+\delta-x)}{(1-z)x+z^2}\frac{\sqrt{x(4-x)}}{2\pi}~dx.$$
Then, as $z\to 0$, 
$$I_{21}= \int \log (4+\delta-x)~d\mathsf{MP}+o(1).$$
\end{lemma}

\begin{proof}
Observe that 
$$I_{21}= \int \log (4+\delta-x)\frac{x}{(1-z)x+z^2}d\mathsf{MP}.$$
Now the proof is completed by noticing that $\frac{x}{(1-z)x+z^2}$ is uniformly bounded for $x\in [0,4]$ and $z\in [0,1]$, and applying DCT. 
\end{proof}

\begin{lemma}
\label{l:i22} Let 
$$I_{22}= \frac{1}{2\pi}\int_{0}^{4} \frac{\log (4+\delta-x)\sqrt{4-x}}{\sqrt{x}((1-z)x+z^2)}~dx.$$
Then as, $z\to 0$ we have 
$$I_{22}= \log (4+\delta) \frac{1}{z}+ \frac{\log (4+\delta)}{2} -2 \int \frac{1}{4+\delta-x}~d\mathsf{MP}-\int_{0}^{4} \frac{\log (4+\delta-x)}{2\pi \sqrt{x(4-x)}}~dx +o(1).$$
\end{lemma}

Proof of Lemma \ref{l:i22} involves some heavy computation so we postpone it for the moment and complete the proof of Proposition \ref{p:inti2}. 

\begin{proof}
We start by observing that 
\begin{eqnarray*}
I_2&=& \int_0^4 \frac{\log (4+\delta-x)}{x} \left[ \frac{x-(1-z)x-z^2}{(1-z)x+z^2} \right]\frac{\sqrt{x(4-x)}}{2\pi}~dx\\
&=&z\int_0^4 \frac{\log (4+\delta-x)}{x} \left[ \frac{(x-z)}{(1-z)x+z^2} \right]\frac{\sqrt{x(4-x)}}{2\pi}~dx\\
&=&z\int_0^4 \frac{\log (4+\delta-x)}{x} \left[1+ \frac{(x-z)}{(1-z)x+z^2}-1\right]\frac{\sqrt{x(4-x)}}{2\pi}~dx\\
&=&z\int_0^4 \frac{\log (4+\delta-x)}{x} \left[1+ \frac{zx-z^2-z}{(1-z)x+z^2}\right]\frac{\sqrt{x(4-x)}}{2\pi}~dx\\
&=&z\int \log (4+\delta-x)~d\mathsf{MP}+z^2\int_0^4 \frac{\log (4+\delta-x)}{x} \left[\frac{x-z-1}{(1-z)x+z^2}\right]\frac{\sqrt{x(4-x)}}{2\pi}~dx\\
&=& z\int \log (4+\delta-x)~d\mathsf{MP}+z^2I_{21}-z^2(z+1)I_{22}.
\end{eqnarray*}
Using Lemma \ref{l:i22} we now have 
$$-(z+1)I_{22}=-\log (4+\delta) \frac{1}{z}+ \left( -\frac{3}{2}\log (4+\delta) + 2 \int \frac{1}{4+\delta-x}~d\mathsf{MP}+\int_{0}^{4} \frac{\log (4+\delta-x)}{2\pi \sqrt{x(4-x)}}~dx\right)+o(1)$$
as $z\to 0$. Using this together with Lemma \ref{l:i21} we get 
\begin{eqnarray*}
I_{2}&=&\left(\int \log (4+\delta-x)~d\mathsf{MP}-\log (4+\delta)\right)z\\ 
&+& \left(-\frac{3}{2}\log (4+\delta)+ 2\int \frac{1}{4+\delta-x} d\mathsf{MP} +\int \log (4+\delta-x)~d\mathsf{MP}+ \int_0^{4} \frac{\log (4+\delta-x)}{2\pi \sqrt{x(4-x)}}~dx\right)z^2\\
&+& o(z^2)
\end{eqnarray*}
as $z\to 0$ completing the proof of the proposition.
\end{proof}

We now move towards the proof of Lemma \ref{l:i22}. For this proof we shall write $I_{22}=I_{221}+I_{222}$ where 
$$I_{221}= \frac{1}{2\pi}\int_{0}^{4} \frac{\log (4+\delta-x)\sqrt{4-x}-2\log (4+\delta)}{\sqrt{x}((1-z)x+z^2)}~dx;~\text{and}$$
$$I_{222}= \frac{2\log (4+\delta)}{2\pi} \int_{0}^{4} \frac{1}{\sqrt{x}((1-z)x+z^2)}~dx.$$
We shall control $I_{221}$ and $I_{222}$ separately in the next two lemmas which put together will immediately imply Lemma \ref{l:i22}. 

\begin{lemma}
\label{l:i221}
With the above notations we have 
$$I_{221}= \frac{\log (4+\delta)}{\pi}-2\int \frac{1}{ (4+\delta-x)}~d\mathsf{MP}-\int_0^{4} \frac{\log (4+\delta-x)}{2\pi \sqrt{x(4-x)}}~dx+o(1)$$
as $z\to 0$.
\end{lemma}

\begin{proof}
Observe first that 
$$\frac{|\log (4+\delta-x)\sqrt{4-x}-2\log (4+\delta)|}{x^{3/2}}$$
is integrable on $[0,4]$ since it behaves like $\frac{1}{\sqrt{x}}$ near $0.$ This and observing that $(1-z)x+z^2=\Omega(x)$ uniformly for $x\in [0,4]$ and for small $z$ implies that DCT is applicable and 
$$ I_{221}=\frac{1}{2\pi}\int_{0}^{4} \frac{\log (4+\delta-x)\sqrt{4-x}-2\log (4+\delta)}{x^{3/2}}~dx +o(1)$$
as $z\to 0$.
Integrating by parts we get 
\begin{align*}
\frac{1}{2\pi}\int_{0}^{4} \frac{\log (4+\delta-x)\sqrt{4-x}-2\log (4+\delta)}{x^{3/2}}~dx &= \left[\frac{-2x^{-1/2}(\log (4+\delta-x)\sqrt{4-x}-2\log (4+\delta))}{2\pi} \right]_0^{4}\\&- \int_{0}^4 \frac{2x^{-1/2}}{2\pi}\left(  \frac{\sqrt{4-x}}{4+\delta-x}+ \frac{\log (4+\delta-x)}{2\sqrt{4-x}}\right)~dx\\
&=\frac{\log (4+\delta)}{\pi}-2\int \frac{1}{\log (4+\delta-x)}~d\mathsf{MP}\\
&-\int_0^{4} \frac{\log (4+\delta-x)}{2\pi \sqrt{x(4-x)}}~dx,
\end{align*}
completing the proof of the lemma.
\end{proof}

\begin{lemma}
\label{l:i222}
With the above notations we have 
$$I_{222}= \frac{\log (4+\delta)}{z}+\frac{\log (4+\delta)}{2}-\frac{\log (4+\delta)}{\pi}+o(1)$$
as $z\to 0$.
\end{lemma}

\begin{proof}
We use the following three facts which are easy to verify: 
\begin{equation}
\label{e:fact1}
\int_{0}^{4} \frac{1}{\sqrt{x}(bx+a)}~dx= \frac{2\arctan \left(2\sqrt{\frac{b}{a}}\right)}{\sqrt{ab}}~\forall a,b>0.
\end{equation}

\begin{equation}
\label{e:fact2}
\frac{\arctan \left(\frac{z}{2\sqrt{1-z}}\right)}{\sqrt{1-z}}=\frac{z}{2}+O(z^2)~\text{as}~z\to 0.
\end{equation}

\begin{equation}
\label{e:fact3}
\frac{1}{\sqrt{1-z}}=1+\frac{z}{2}+O(z^2)~\text{as}~z\to 0.
\end{equation}

Using \eqref{e:fact1} with $a=z^2$ and $b=(1-z)$ gives

\begin{eqnarray*}
I_{222} &=&  \frac{2\log (4+\delta)}{2\pi} \frac{2 \arctan \left( 2\sqrt{\frac{1-z}{z^2}} \right)}{\sqrt{z^2(1-z)}}\\
&=& \frac{4\log (4+\delta)}{2\pi} \frac{\left( \frac{\pi}{2}- \arctan \frac{z}{2\sqrt{1-z}}  \right)}{z\sqrt{1-z}}\\
&=& \frac{\log (4+\delta)}{z\sqrt{1-z}}- \frac{4\log (4+\delta)}{2\pi} \frac{\arctan \frac{z}{2\sqrt{1-z}}}{z\sqrt{1-z}}.
\end{eqnarray*}

Using this together with \eqref{e:fact2} and \eqref{e:fact3} we get 
\begin{eqnarray*}
I_{222} &=& \frac{\log (4+\delta)}{z} \left[1+\frac{z}{2}+O(z^2)\right] -\frac{4\log (4+\delta)}{2\pi z} \left[ \frac{z}{2}+O(z^2)\right]\\
&=&  \frac{\log (4+\delta)}{z}+\frac{\log (4+\delta)}{2}-\frac{\log (4+\delta)}{\pi}+O(z)
\end{eqnarray*}
as $z\to 0$ completing the proof of the lemma.

\end{proof}

We are now ready to prove Theorem \ref{t:ratecompare} which directly implies Proposition \ref{curvature}.
Let us recall the set up: $n$ is a large integer and let $0\leq c \leq n$ be in $\N$. Set $m_1=n+c$, $n_1=n-c$ and $y=\frac{n_1}{m_1} \leq 1$. Define, 
\begin{equation}
\label{e:a0b0}
A_0 =\log \frac{Z_{m_1-1,n_1-1}}{Z_{m_1,n_1}}, \qquad B_0=\log \frac{Z_{n-1,n-1}}{Z_{n,n}},
\end{equation}

\begin{equation}
\label{e:a1b1}
A_1 = \left(2n_1 \int \log(4+\hat \delta-x){\rm d}\mathsf{MP}_y \right), \qquad B_1=\left(2n \int \log(4+ \delta-x){\rm d}\mathsf{MP} \right),
\end{equation}

\begin{equation}
\label{e:a2b2}
A_2 = -n_1 \int x{\rm d}\mathsf{MP}_y, \qquad B_2=-n \int x{\rm d}\mathsf{MP},
\end{equation}

\begin{equation}
\label{e:a3b3}
A_3 = (m_1-n_1) \log (4+\hat \delta), \qquad B_3=-(n-n) \log (4+ \delta),
\end{equation}

\begin{equation}
\label{e:a4b4}
A_4 = -m_1 (4+\hat \delta), \qquad B_3=-n(4+ \delta),
\end{equation}
where $\hat \delta$ is defined by $(4+\delta)n=(4+\hat \delta)m_1$. Our main objective is to prove the following result. 

\begin{theorem}
\label{t:ratecompare}
For $\delta>0$, with the notations as above, we have
$$\sum_{i=0}^{4} (A_{i}-B_{i})=-\beta_{\delta} \frac{c^2}{n} +o\left(\frac{c^2}{n}\right)$$
as $\frac{c}{n}\to 0$ for some $\beta_{\delta}>0$. 
\end{theorem}
We note down some preparatory facts before starting with the proof of Theorem \ref{t:ratecompare} starting by recalling \eqref{a0}-\eqref{a4}.

\begin{align}
\label{a01}
A_{0}-B_1&=\log \frac{Z_{m_1-1,n_1-1}}{Z_{m_1,n_1}}-\log \frac{Z_{n-1,n-1}}{Z_{n,n}}\\
\label{a11}
A_1-B_1&= \left(2n_1 \int \log(4+\hat \delta-x){\rm d}\mathsf{MP}_y \right)-\left(2n \int \log(4+ \delta-x){\rm d}\mathsf{MP} \right),\\
\label{a21}
A_2-B_2&=\left(-n_1 \int x{\rm d}\mathsf{MP}_y\right)+\left(n \int x{\rm d}\mathsf{MP}\right)=c~\text{by \eqref{e:mpmean}},\\
\label{a31}
A_3-B_3&=2c \log (4+\hat \delta),\\
\label{a41}
A_4-B_4 &= -m_1 (4+\hat \delta)+n (4+ \delta)=0~\text{by definition of}~\hat \delta.
\end{align}

Our main job will be to control $A_1-B_1$ which is done using Theorem \ref{t:intest}, but first let us go about the easier tasks of controlling $A_0-B_0$ and $A_3-B_3$. 

\begin{lemma}
\label{l:0}
We have $A_0-B_0=c-\frac{6c^2}{n}+o(\frac{c^2}{n})$ as $\frac{c}{n}\to 0$. 
\end{lemma}
\begin{proof}Recall  Lemma \ref{l:partition} i.e.,
$\log \frac{Z_{m_1,n_1}}{Z_{m_1-1,n_1-1}}= -2n_1-m_1+n_1\log \frac{n_1}{m_1}+ O(1).$ Thus,

\begin{align*}
\log(\frac{Z_{m_1,n_1}}{Z_{m_1-1,n_1-1}})
&=-3n+c+ (n-c)\log (1-\frac{2c}{n+c})+O(1)\\
&=-3n+c+ (n-c)(-\frac{2c}{n+c}+\frac{1}{2}\left(\frac{2c}{n+c}\right)^2+O(\frac{c^3}{n^3}))\\
&=-3n+c+ (n-c)(-\frac{2c}{n+c}+\frac{1}{2}\left(\frac{2c}{n+c}\right)^2)+O(\frac{c^3}{n^2})\\
&=-3n-c +\frac{6c^2}{n}+O(\frac{c^3}{n^2}).
\end{align*}
\end{proof}
\begin{lemma}
\label{l:3}
We have $A_3-B_3= 2c\log (4+\delta)-\frac{2c^2}{n}+o(\frac{c^2}{n})$ as $\frac{c}{n}\to 0$. 
\end{lemma}

\begin{proof}
Observe that, by definition of $\hat \delta$ we have 
\begin{equation}
\log (4+\hat\delta)=\log (4+\delta) +\log \left(\frac{1}{1+\frac{c}{n}}\right)= \log (4+\delta)-\frac{c}{n}+o(\frac{c}{n})
\end{equation}
as $\frac{c}{n}\to 0$. The lemma follows.
\end{proof}

We now provide the remaining details of the proof of Theorem \ref{t:ratecompare}.
\begin{proof}[Proof of Theorem \ref{t:ratecompare}]
Observe that 
\begin{eqnarray*}
A_1-B_1 &=& 2n\left(\int \log(4+\hat \delta-x){\rm d}\mathsf{MP}_y - \int \log(4+ \delta-x){\rm d}\mathsf{MP} \right)\\
&-&2c\left(\int \log(4+\hat \delta-x){\rm d}\mathsf{MP}_y -\int \log(4+ \delta-x){\rm d}\mathsf{MP} \right)\\
&-& 2c  \int \log(4+ \delta-x){\rm d}\mathsf{MP}.
\end{eqnarray*}
Notice now that 
$$y=\frac{n-c}{n+c}=1-\frac{2c}{n}+ \frac{2c^2}{n^2}+O(\frac{c^3}{n^3})$$
as $\frac{c}{n}\to 0$. Using the notation from Theorem \ref{t:intest}, we have 
$$z=1-\sqrt{y}=\frac{c}{n} -\frac{c^2}{2n^2}+O(\frac{c^3}{n^3}).$$
Using this together with Theorem \ref{t:intest}, we get, as $\frac{c}{n}\to 0$ that
$$A_1-B_1=c\left(2A-2\int \log(4+ \delta-x){\rm d}\mathsf{MP}\right)+\frac{c^2}{n}{\left( -3A+2B\right)}+ o(\frac{c^2}{n}),$$
where $A$ and $B$ are as in Theorem \ref{t:intest}. Putting this together with \eqref{a21}, \eqref{a41}, Lemma \ref{l:0} and Lemma \ref{l:3} we get 
$$\sum_{i=0}^{4} (A_i-B_i)= \alpha_{\delta}c- \beta_{\delta} \frac{c^2}{n}+o(\frac{c^2}{n})$$
where $\alpha_{\delta}= 2+2\log(4+\delta)+2A-2\int \log(4+ \delta-x){\rm d}\mathsf{MP}=0$ by plugging in the value of $A$ from Theorem \ref{t:intest}; and, 
recalling $$A= -1+ \int \log (4+\delta-x)~d\mathsf{MP}-\log (4+\delta);$$
$$B=-{\frac12}-\frac{3}{2}\log (4+\delta)+ (\frac{1}{2}(2+\delta)+2) \int \frac{1}{4+\delta-x} d\mathsf{MP} +\int \log (4+\delta-x)~d\mathsf{MP}+ \int_0^{4} \frac{\log (4+\delta-x)}{2\pi \sqrt{x(4-x)}}~dx,$$  
we get 
\begin{eqnarray*}
-\beta_{\delta}&=&-6+ (-3A+2B)-2=-6-\int \log(4+ \delta-x){\rm d}\mathsf{MP}\\
&+&(6+\delta) \int \frac{1}{4+\delta-x} d\mathsf{MP} +2\int_0^{4} \frac{\log (4+\delta-x)}{2\pi \sqrt{x(4-x)}}~dx.
\end{eqnarray*}
Computation of standard integrals  show that $-\beta_0=-4.$
Simple analysis also shows that $-\beta_{\infty}=-5.$
Differentiation under integral sign in $\delta,$ shows that 
\begin{align*}
-\beta'_{\delta}&= -\int \frac{(6+\delta)}{(4+\delta-x)^2} d\mathsf{MP}+2\int \frac{1}{2\pi(4+\delta-x)\sqrt{x(4-x)}},\\
&= - \frac{(6+\delta)}{(4+\delta)^{3/2}\sqrt{\delta}}+\frac{1}{\sqrt{(4+\delta)\delta}}<0.
\end{align*}
\end{proof}

\bibliography{slowbond}
\bibliographystyle{plain}

\end{document}